\documentclass[12pt]{article}

\usepackage[margin=1in]{geometry}
\usepackage{amsmath,amssymb,amsfonts,amsthm,mathtools,mathrsfs,array,booktabs,tabularx}
\usepackage{latexsym,stmaryrd,graphicx,cancel,array,setspace}
\usepackage[dvipsnames,table]{xcolor}
\usepackage{algorithm}
\usepackage{algpseudocode}
\usepackage{tikz,tikz-cd}
\usetikzlibrary{arrows,knots,calc,positioning,arrows.meta}
\usepackage{ytableau}
\usepackage[shortlabels]{enumitem}
\usepackage{float}
\usepackage{longtable}
\usepackage{microtype}
\usepackage{placeins}
\usepackage{needspace}
\usepackage[colorlinks=true,linkcolor=blue,citecolor=blue,urlcolor=blue]{hyperref}
\usetikzlibrary{
    arrows.meta,
    decorations.pathmorphing,
    shapes.symbols
}

\newtheorem{thm}{Theorem}[section]
\newtheorem*{thm*}{Theorem}
\newtheorem{cor}[thm]{Corollary}
\newtheorem*{cor*}{Corollary}
\newtheorem{lem}[thm]{Lemma}
\newtheorem*{lem*}{Lemma}
\newtheorem{prop}[thm]{Proposition}
\newtheorem*{prop*}{Proposition}
\newtheorem{prob}{Problem}

\newtheorem{mainthm}{Theorem}

\newtheorem{maincor}[mainthm]{Corollary}

\theoremstyle{definition}
\newtheorem{defn}[thm]{Definition}
\newtheorem*{defn*}{Definition}

\newtheorem*{conjecture*}{Conjecture}

\newtheorem*{condition*}{Condition}
\newtheorem*{assumption*}{Assumption}

\theoremstyle{remark}
\newtheorem{rem}[thm]{Remark}
\newtheorem*{rem*}{Remark}
\newtheorem{example}[thm]{Example}

\newtheorem{problem}[thm]{Problem}
\newtheorem*{problem*}{Problem}

\newcommand{\Ht}{\widetilde{H}}

\newcommand{\Z}{\mathbb{Z}}

\newcommand{\C}{\mathbb{C}}
\newcommand{\xn}{\mathbf{x}_n}
\newcommand{\yn}{\mathbf{y}_n}
\newcommand{\xv}{\mathbf{x}}
\newcommand{\yv}{\mathbf{y}}

\newcommand{\mxy}[2]{\xv^{#1}\yv^{#2}}

\newcommand{\symm}{\mathfrak{S}}

\newcommand{\CST}{\mathrm{CST}}
\newcommand{\SYT}{\mathrm{SYT}}

\newcommand{\RowLat}{\mathrm{Lat}}

\newcommand{\Basis}{\mathcal{B}}

\DeclareMathOperator{\maj}{maj}
\DeclareMathOperator{\inv}{inv}
\DeclareMathOperator{\Inv}{Inv}
\DeclareMathOperator{\Des}{Des}
\DeclareMathOperator{\rowht}{ht}

\DeclareMathOperator{\tp}{top}
\DeclareMathOperator{\Hilb}{Hilb}
\DeclareMathOperator{\Frob}{Frob}

\newcommand{\invvec}{\operatorname{invvec}}
\newcommand{\rowvec}{\operatorname{rowvec}}

\newcommand{\rowsort}{\operatorname{rowsort}}
\newcommand{\Lat}{\mathrm{Lat}}

\newcommand{\blankcell}{\phantom{\bullet}}

\renewcommand{\arraystretch}{1.2}

\title{Bases for the Top-Degree Garsia--Haiman Module}
\author{Raymond Chou}

\AtBeginDocument{%
  \setlength{\abovedisplayskip}{10pt plus 3pt minus 6pt}%
  \setlength{\belowdisplayskip}{10pt plus 3pt minus 6pt}%
  \setlength{\abovedisplayshortskip}{0pt plus 3pt}%
  \setlength{\belowdisplayshortskip}{5pt plus 3pt minus 3pt}%
}

\begin{document}

\maketitle

\begin{abstract}
    We introduce two new monomial bases for the top $t$-degree component of the Garsia--Haiman module $V_\mu$ (first introduced in \cite{garsia1993graded}) for every partition $\mu$. Both bases are indexed by column-strict standard fillings of shape $\mu$; the first has exponents governed by entrywise refinements of Haglund's combinatorial statistics for the modified Macdonald polynomials $\Ht_\mu(X;q,t)$ \cite{haglund2005combinatorial}. The second basis has exponents governed by an insertion algorithm of Bhattacharya--Ratheesh--Viswanath \cite{bhattacharya2024whittaker}, and gives a basis realization of the ``fermionic formula" of the $m_{1^n}$-coefficient of the $q$-Whittaker polynomial. Our proof uses orbit harmonics, and is inspired by the original approach of Garsia and Haiman \cite{garsia1996natural}.
\end{abstract}

\section{Introduction}

Let $\xn := \{x_1,\dots,x_n\},\yn := \{y_1,\dots,y_n\}$, and write $R := \C[\xn,\yn]$ for the polynomial ring in $2n$ variables. We have a natural duality pairing
$$ \langle f, g \rangle := \text{ constant term of }f(\partial x_1,\dots,\partial x_n ; \partial y_1,\dots,\partial y_n) \cdot g(x_1,\dots,x_n;y_1,\dots,y_n)$$
where $\cdot$ denotes acting by a linear extension of partial derivative operators. Given a partition $\mu \vdash n$, Garsia and Haiman \cite{garsia1993graded} define a generalized Vandermonde determinant $\Delta_\mu$; its \emph{annihilating ideal} is defined to be
$$J_\mu := \{ f \in \C[\xn,\yn] : f \cdot \Delta_\mu = 0 \}$$
and we define the \emph{Garsia--Haiman module} to be $V_\mu := \C[\xn,\yn]/J_\mu$. Garsia and Haiman \cite{garsia1993graded} originally constructed the doubly-graded $S_n$-module $V_\mu$ to act as a representation-theoretic model for the \emph{modified Macdonald polynomial} $\Ht_\mu(X;q,t)$, to be defined below. More precisely, $\Frob_{q,t}(V_\mu) = \Ht_\mu(X;q,t)$. This was implied by Haiman's \emph{$n!$ theorem} \cite[Theorem~3.2]{haiman2001hilbert}, which states that for any partition $\mu \vdash n$, $\dim_\C(V_\mu) = n!$, which in turn established the Schur-positivity of the $\Ht_\mu(X;q,t)$, a very important open problem at the time.

Haiman's proof consisted of a deep analysis of the \emph{Hilbert Scheme of points in $\C^2$}, denoted $\Hilb_n(\C^2)$. The rings $V_\mu$ were interpreted as coordinate rings of the scheme-theoretic fibers in the \emph{isospectral Hilbert Scheme} $X_n \to\Hilb_n(\C^2)$ over the torus-fixed points indexed by $\mu$ \cite[Proposition~3.7.3]{haiman2001hilbert}, and the result was obtained by showing $X_n$ was Gorenstein, hence Cohen--Macaulay, and that the natural projection $X_n \to \Hilb_n(\C^2)$ is finite and flat of degree $n!$ \cite[Theorem~3.1 and Proposition~3.7.4]{haiman2001hilbert}. Haiman's proof, however, did not provide an explicit basis, which leaves the following open problem:

\begin{prob}\label{problem1}
    Given a partition $\mu \vdash n$, construct an explicit basis for the Garsia--Haiman module $V_\mu$, preferably in terms of combinatorial data associated to $\mu$.
\end{prob}

In 2004, Jim Haglund ``found water on Mars" (as Garsia put it \cite{haglund2006genesis}) by discovering a combinatorial formula for the modified Macdonald polynomials \cite{haglund2004combinatorial}, proven the following year with Haiman and Loehr \cite{haglund2005combinatorial}. This was a leap forward, as it now gave a natural combinatorial construction that potential basis candidates could be verified against. This led to partial progress for the basis problem: for $\mu = (n-k+1,1^{k-1})$ a hook shape, Adin--Remmel--Roichman have established a family of bases \cite{adin2014combinatorics}, including the $k$-th Haglund basis, directly based on Haglund's statistics. Separately, the two-column case $\mu_1 \leq 2$ has a ``kicking basis" due to Assaf--Garsia \cite{assaf2009kicking}, based off the orbit harmonics construction, recalled in detail below.

The purpose of this paper is to give a partial solution to Problem \ref{problem1}. Writing $n(\mu) = \sum_i (i-1)\mu_i$, we may decompose $V_\mu$ into its graded components:

\begin{equation}\label{eq:ghgrading}
    V_\mu \cong \bigoplus_{j = 0}^{n(\mu)} \bigoplus_{i = 0}^{n(\mu')} V_\mu^{(i,j)}
\end{equation}

where $i$ and $j$ track the degrees of the $x$- and $y$-variables, respectively. Let us now restrict our attention to when $j = n(\mu)$. Define

$$V_\mu^{\tp} := \bigoplus_{i=0}^{n(\mu')} V_\mu^{(i,n(\mu))}$$

so that the $t$-degree is maximal. Our theorem describes bases for $V_\mu^{\tp}$.

The first basis, first conjectured in the author's thesis, is indexed by column-strict standard fillings of shape \(\mu\), with exponents governed by the top-degree specialization of the Haglund--Haiman--Loehr statistics; they are recalled in detail in Section~\ref{sec:background}. Write \(\CST_\mu\) for this set of fillings and, for \(\sigma\in\CST_\mu\), write
\[
    \mathfrak m_\sigma
    :=\mxy{\invvec(\sigma)}{\rowvec(\sigma)}
\]
Here \(\rowvec(\sigma)\) records the row heights of the entries, while \(\invvec(\sigma)\) is a tuple of nonnegative integers tracking the middle entry of Haglund's inversion triples; we give a detailed definition in Section \ref{sec:background}. 

\begin{example}
    Let $\mu = (3,2)$, and consider the filling $\sigma \in \CST_{(3,2)}$ given by
    $$
    \begin{ytableau}
        5 & 4 \\
        1 & 3 & 2
    \end{ytableau}
    $$
    The inversion triples are $(-\infty,2,3)$ and $(1,4,5)$, so that
    \begin{gather*}
        \invvec(\sigma) = (0,1,0,1,0), \qquad
        \rowvec(\sigma) = (0,0,0,1,1) \\
        \text{and}\qquad \mathfrak{m}_\sigma = x_2x_4y_4y_5
    \end{gather*}
\end{example}

\begin{mainthm}[Haglund basis; Conjecture~5.4.5 of~\cite{chou2023thesis}]\label{thma:intro}
    The monomials
    \[
        \Basis_\mu^{\mathrm{top}} := \{\mathfrak m_\sigma:\sigma\in\CST_\mu\}
    \]
    descend to a vector-space basis for the top \(t\)-degree component \(V_\mu^{\tp}\) of the Garsia--Haiman module.
\end{mainthm}

The proof consists of two main ingredients. First, we construct a family of linear functionals $\mathcal{A}_\tau$ indexed by column-strict fillings $\tau \in \CST_\mu$, which are triangular by $\yv$-monomials. Then, we construct for each $\sigma \in \CST_\mu$ a certain \emph{Newton polynomial} $\mathcal{N}_\sigma(\xv)$ whose top-degree term is precisely $\xv^{\invvec(\sigma)}$. We then evaluate the functionals $\mathcal{A}_\tau$ on $\mathcal{N}_\sigma(\xv)\yv^{\rowvec(\sigma)}$ and show this matrix is block triangular with invertible diagonal blocks. This proves that the classes of the $\mathcal{N}_\sigma(\xv)\yv^{\rowvec(\sigma)}$ are linearly independent in $\mathcal Q_\mu$; a counting argument, followed by passage to the $\xv$-associated graded, finishes the proof.

For a fixed row word $\beta$, the collection of monomials $\xv^\eta\yv^\beta \in \Basis_\mu^{\mathrm{top}}$ turns out to have $\xv$-degree generating function equal to a product of $q$-numbers. We construct a second basis based on Bhattacharya--Ratheesh--Viswanath's insertion algorithm, and discuss this factorization with respect to the $q$-Whittaker polynomial in Section \ref{sec:fermionic}. This gives the following second basis theorem:

\begin{mainthm}[Flag basis]\label{thmb:intro}
    There is a second monomial basis for $V_\mu^{\mathrm{top}}$ of the form
    $$ \widehat{\mathcal{B}}_\mu = \bigsqcup_{\beta \in \mathrm{Lat}_\mu} \{\xv^\zeta \yv^\beta : \zeta\in\Z_{\ge0}^n,\ 0 \leq \zeta_a < g_a(\beta)\text{ for all }a\}$$
    where $\mathrm{Lat}_\mu$ denotes the set of $\mu$-\emph{lattice words}, and $g_a(\beta)$ are given by
    $$ \lambda_r^{(a)}(\beta) = \#\{i \leq a : \beta_i = r-1\},\qquad g_a(\beta) = \lambda^{(a)}_{\beta_a + 1}(\beta) - \lambda^{(a)}_{\beta_a + 2}(\beta)$$
    As a consequence, this gives a basis realization of the ``fermionic formula" of the Hilbert series for $V_\mu^{\mathrm{top}}$, where $[m]_q:=1+q+\cdots+q^{m-1}$:
    $$ \Hilb_q(V_\mu^{\mathrm{top}}) = \sum_{\beta \in \mathrm{Lat}_\mu} \prod_{a = 1}^n [g_a(\beta)]_q$$
\end{mainthm}

We then consider a lexicographic filtration by $\yv$-monomials that is compatible with both bases; this gives a graded count of the monomials in the Haglund basis $\Basis_\mu^{\mathrm{top}}$ with a fixed $\yv$-monomial:

\begin{maincor}[Row-word enumeration]\label{corc:intro}
    For every $\beta\in\mathrm{Lat}_\mu$, we have
    \[
        \sum_{\substack{\sigma\in\CST_\mu\\
                        \rowvec(\sigma)=\beta}}
            q^{|\invvec(\sigma)|}
        =\prod_{a=1}^{n}[g_a(\beta)]_q.
    \]
\end{maincor}

The paper is organized as follows.  Section~\ref{sec:background} recalls the
Haglund statistics, the modified Macdonald polynomial, and the
Garsia--Haiman module.  Section~\ref{sec:orbit-newton} develops the
orbit-harmonic and Newton interpolation argument and proves Theorem~A.  The
final sections discuss the fermionic formula, $q$-Whittaker polynomials,
and further directions.

\subsection{Acknowledgements}

The author would like to thank Erik Carlsson, Jim Haglund, and Brendon Rhoades (in no particular order) for helpful conversations. The author would also like to thank the barista in Gothenburg, Sweden whose coffee spawned many important ideas for this paper.

\subsection{AI usage statement}

The author used ChatGPT in the editing and preparation of this draft, including many of the figures. Additionally, ChatGPT was used extensively to produce code and to test/verify the results, and in organizing the ideas of this paper. All mathematical content is the sole responsibility of the author.

\section{Background}\label{sec:background}

In this section, we briefly record our combinatorial conventions, as well as recall necessary definitions for the proof. For more background on Garsia--Haiman modules, the reader is referred to Haiman's notes \cite{haiman1999geometry}. For more background on modified Macdonald polynomials, the reader is referred to \cite{haglund2005combinatorial}.

\subsection{Combinatorial preliminaries}

Let \(\mu\vdash n\), and let \(D_\mu\) denote its Ferrers diagram in French notation:
\[
    D_\mu=\{(c,h):0\le h<\ell(\mu),\ 0\le c<\mu_{h+1}\}.
\]
We use the variables $c$ and $h$ to denote the \emph{column} and the \emph{height} respectively of the box. A \emph{filling} of shape \(\mu\) is a map \(\sigma:D_\mu\to \Z_{>0}\); it is \emph{standard} if its entries are \([n]\).  We write \(\CST_\mu\) for the set of standard fillings whose columns strictly decrease when read from top to bottom; equivalently, if \(u\) is directly above \(v\), then \(\sigma(u)>\sigma(v)\).

\begin{defn}[descents, arms, legs, and row heights]\label{def:descents-row-heights}\cite{haglund2005combinatorial}
Let \(\sigma\) be a filling of \(\mu\).  A \emph{descent} is a box \(u\) directly above another box \(v\) such that \(\sigma(u)>\sigma(v)\).  We write \(\Des(\sigma)\) for the set of descents.  The \emph{leg} \(\ell(u)\) of a box \(u\) is the number of boxes strictly above \(u\) in its column, and the \emph{arm} $a(u)$ is the number of boxes strictly to the right in its row.  The \emph{height} \(\rowht(u)\) of \(u=(c,h)\) is \(h\).  If \(\sigma(u)=a\), we also write \(\rowht_\sigma(a)=\rowht(u)\).
\end{defn}

The \emph{major index} of a filling $\sigma$ is
\[
    \maj(\sigma):=\sum_{u\in\Des(\sigma)}(\ell(u)+1).
\]
For \(\sigma\in\CST_\mu\), every box not in the bottom row is a descent, so that \(\maj(\sigma)=n(\mu)\), the largest possible \(t\)-degree. For our analysis of the top-degree case, we therefore only consider column-strict fillings, which have maximal major index.

\begin{defn}[Haglund's inversion statistic]
    Let $\sigma$ be a $\mu$-filling. An \emph{inversion pair} is a pair of boxes $(u,v)$ with $a =\sigma(u) > \sigma(v) = b$ in one of the following configurations:
    
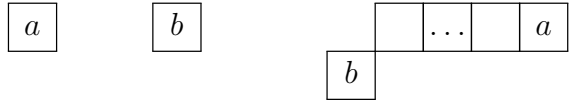
\[
    \begin{ytableau}
        a & \none & \none & b
    \end{ytableau}
    \qquad\qquad
    \begin{ytableau}
        \none & \blankcell & \dots & \blankcell & a \\
        b
    \end{ytableau}
    \]
    where in the configuration on the right, $b$ occurs in the row directly below $a$. The set of inversions is defined to be $\Inv_{\mathrm{pair}}(\sigma)$, and we define the inversion statistic $\inv$ to be
    $$ \inv(\sigma) := |\Inv_{\mathrm{pair}}(\sigma)| - \sum_{u \in \Des(\sigma)} a(u)$$
\end{defn}

\begin{defn}\label{def:top-invtriples}
Let \(\sigma\in\CST_\mu\).  A \emph{top-degree inversion triple} is a configuration of entries
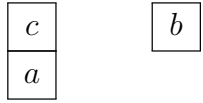
\[
\begin{ytableau}
    c & \none & \none & b \\
    a
\end{ytableau}
\]
with \(c>b>a\).  If \(b\) and \(c\) lie in the bottom row, the lower entry \(a\) is absent, and we write the triple as \((-\infty,b,c)\).  The set of all such triples is denoted \(\Inv(\sigma)\).
\end{defn}

One recognizes that this is simply Haglund's inversion statistic after restricting to the top \(t\)-degree part; see \cite[Section~2, Equations~(15)--(17)]{haglund2005combinatorial}.  For our purposes, we will need to remember which middle entry contributes each inversion.  Define
\[
    d_i(\sigma):=\#\{(a,b,c)\in\Inv(\sigma):b=i\}.
\]
Then set
\[
    \invvec(\sigma):=(d_1(\sigma),\ldots,d_n(\sigma)),
    \qquad
    \rowvec(\sigma):=(\rowht_\sigma(1),\ldots,\rowht_\sigma(n)).
\]
The monomial attached to \(\sigma\) is
\begin{equation}\label{def:xinvtymajt}
    \mathfrak m_\sigma
    :=\xv^{\invvec(\sigma)}\yv^{\rowvec(\sigma)}
    =\bigg(\prod_{(a,b,c)\in\Inv(\sigma)}x_b\bigg)\bigg(
      \prod_{i=1}^n y_i^{\rowht_\sigma(i)}\bigg).
\end{equation}

\begin{example}\label{ex:topdegbasis22}
If \(\mu=(2,2)\), then \(\CST_\mu\) consists of the following six fillings:
\begin{center}
\begingroup\setlength{\tabcolsep}{4.8pt}
\begin{tabular}{|c||c|c|c|c|c|c|}
\hline
\(\sigma\) &
\begin{ytableau}3&4\\1&2\end{ytableau} &
\begin{ytableau}4&3\\1&2\end{ytableau} &
\begin{ytableau}3&4\\2&1\end{ytableau} &
\begin{ytableau}4&3\\2&1\end{ytableau} &
\begin{ytableau}2&4\\1&3\end{ytableau} &
\begin{ytableau}4&2\\3&1\end{ytableau} \\
\hline
\(\Inv(\sigma)\) &
\(\varnothing\) &
\(\{(1,3,4)\}\) &
\(\{(-\infty,1,2)\}\) &
\(\{(-\infty,1,2),(2,3,4)\}\) &
\(\varnothing\) &
\(\{(-\infty,1,3)\}\) \\
\hline
\(\mathfrak m_\sigma\) &
\(y_3y_4\) &
\(x_3y_3y_4\) &
\(x_1y_3y_4\) &
\(x_1x_3y_3y_4\) &
\(y_2y_4\) &
\(x_1y_2y_4\) \\
\hline
\end{tabular}
\endgroup
\end{center}
The bigraded degree generating function is \(t^2(2+3q+q^2)\), agreeing with the bigraded Hilbert series of \(V_{(2,2)}^{\tp}\).
\end{example}

\subsection{Modified Macdonald polynomials}

We denote the \emph{modified Macdonald polynomial} by \(\Ht_\mu(X;q,t)\).  The Haglund--Haiman--Loehr formula \cite[Theorem~2.2, Equation~(14)]{haglund2005combinatorial} expresses it as a weighted sum over fillings of \(\mu\):
\[
    \Ht_\mu(X;q,t)
    =\sum_{\sigma:\mu\to\Z_{>0}}
        q^{\inv(\sigma)}t^{\maj(\sigma)}x^\sigma,
\]
where \(x^\sigma=\prod_{u\in D_\mu}x_{\sigma(u)}\).  For the $m_{1^n}$-coefficient in top \(t\)-degree, the fillings which appear are precisely the column-strict standard fillings \(\CST_\mu\), and the \(q\)-weight is recorded by the inversion tuple above.  Thus
\[
    [t^{n(\mu)}]\Ht_\mu(X;q,t)\big|_{m_{1^n}}
    =\sum_{\sigma\in\CST_\mu}q^{|\invvec(\sigma)|}.
\]
We will use this identity for the Hilbert-series of $V_\mu^{\mathrm{top}}$ to complete the basis proof.

\subsection{\texorpdfstring{$q$}{q}-Whittaker polynomials}\label{subsec:q-whittaker-background}

Let $P_\mu(X;q,t)$ be the $P$-Macdonald polynomials, defined in \cite{macdonald1988new}. They are related to the modified Macdonald polynomials via the following transformation; see \cite[Chapter~VI, Section~8]{macdonald1995symmetric} and \cite[Equation~(60)]{haglund2005combinatorial}:

$$ P_\mu[X;q,t] = \frac{t^{n(\mu)} \widetilde{H}_\mu[X(1-t);q,t^{-1}]}{\prod_{u \in D_\mu}\big(1 - q^{a(u)}t^{\ell(u)+1}\big)}$$

where the brackets denote a \emph{plethystic substitution}. Classically, the $q$-Whittaker polynomials are the $t = 0$ specializations of $P_\mu$ \cite{macdonald1995symmetric}:

$$ W_\mu(X;q) := P_\mu(X;q,0)$$

They are generalizations of the classical special functions considered by Whittaker and Jacquet \cite{whittaker1903hypergeometric,jacquet1967whittaker}. They appear in a wide variety of contexts; quantum group analogues of Whittaker functions and their relationship with certain $q$-deformed Toda operators were developed by Etingof \cite{etingof1999whittaker} and Sevostyanov \cite{sevostyanov2000whittaker}. Explicit formulae for the related ``$q$-deformed $\mathfrak{gl}_m$-Whittaker functions" were given by Gerasimov--Lebedev--Oblezin \cite{gerasimov2010whittakerI}, and normalized versions of these functions were identified with $W_\mu(X;q)$ in \cite{gerasimov2011whittakerIII}.

For our purposes, the $q$-Whittaker polynomial turns out to be the top $t$-degree of the modified Macdonald polynomials:

\begin{equation}\label{eq:q-whittaker-top-coefficient}
    W_\mu(X;q)
       =[t^{n(\mu)}]\Ht_\mu(X;q,t);
\end{equation}
see F. Bergeron's 2020 survey of $q$-Whittaker polynomials \cite[Section~3]{bergeron2020whittaker}. Consequently,
\[
    \Hilb_q(V_\mu^\mathrm{top})
       =[m_{1^n}]W_\mu(X;q),
\]
where $[m_{1^n}]$ denotes coefficient extraction in the
monomial symmetric-function basis.

\subsection{Garsia--Haiman modules and the \texorpdfstring{\(n!\)}{n!} theorem}\label{subsec:gh-background}

Let \(R=\C[\xn,\yn]\).  Associate to \(\mu\) the set of lattice points
\[
    D_\mu=\{(p,q):0\le q<\ell(\mu),\ 0\le p<\mu_{q+1}\}.
\]
For example, if \(\mu=(2,2,1)\), then
\[
\begin{ytableau}
    \scriptstyle (0,2) \\
    \scriptstyle (0,1) & \scriptstyle (1,1) \\
    \scriptstyle (0,0) & \scriptstyle (1,0)
\end{ytableau}.
\]
Choose any ordering \((p_1,q_1),\ldots,(p_n,q_n)\) of \(D_\mu\), and define the generalized Vandermonde determinant
\[
    \Delta_\mu:=\det\bigl(x_i^{p_j}y_i^{q_j}\bigr)_{1\le i,j\le n}.
\]
The order of the boxes changes \(\Delta_\mu\) only by sign.

The dual Garsia--Haiman module is the span of all partial derivatives of \(\Delta_\mu\):
\[
    V_\mu' := \C[\partial x_1,\ldots,\partial x_n,
                    \partial y_1,\ldots,
                    \partial y_n]\cdot \Delta_\mu.
\]
Equivalently, let
\[
    J_\mu:=\{f\in R:f(\partial \xv,\partial \yv)\Delta_\mu=0\}
\]
be the annihilating ideal of \(\Delta_\mu\).  Then
\[
    V_\mu:=R/J_\mu
\]
is the Garsia--Haiman module.  We grade \(V_\mu\) by \(x\)-degree and \(y\)-degree and write \(V_\mu^{(i,j)}\) for the bihomogeneous piece.  The piece studied in this paper is
\[
    V_\mu^{\tp}:=\bigoplus_i V_\mu^{(i,n(\mu))},
\]
the maximal \(y\)-degree piece.

Haiman's \(n!\) theorem \cite[Theorem~3.2]{haiman2001hilbert} states that \(\dim_\C V_\mu=n!\), and moreover the bigraded Frobenius characteristic of \(V_\mu\) is \(\Ht_\mu(X;q,t)\).  Therefore the Hilbert series of \(V_\mu^{\tp}\) is the top \(t\)-degree coefficient of \([m_{1^n}]\Ht_\mu\).  This is why the monomials in Equation~\eqref{def:xinvtymajt} have the correct cardinality and grading to form a basis, once spanning is proved.

\subsection{The orbit harmonics quotient}\label{subsec:filtered-orbit-harmonics}

We briefly recall the orbit harmonics philosophy, before specializing it for the uses of this paper.

The \emph{orbit harmonics deformation} associates a graded algebra to a finite set of points. Orbit harmonics has been the subject of many recent papers \cite{liu2025involution,oh2022cyclic,reiner2026ehrhart,reineke2023zonotopal,rhoades2024viennot,rhoades2025big}, including in the study of cohomology rings of interesting combinatorial varieties, see \cite{chou2026torus,CMR,griffin2021ordered}.

Let $S = \C[z_1,\dots,z_d]$, and $\mathcal{O} \subseteq \C^d$ be a nonempty finite set. We have the vanishing ideal
$$ I(\mathcal{O}) := \{f \in S: f(p) = 0\text{ for every }p \in \mathcal{O}\}$$
For a nonzero polynomial $f$, let $\tau(f)$ denote its highest-degree homogeneous component, and define
$$ \mathrm{gr}I(\mathcal{O}) := \langle \tau(f) : 0 \neq f \in I(\mathcal{O})\rangle$$

The \emph{orbit harmonics quotient} $S/\mathrm{gr}I(\mathcal{O})$ is the associated graded ring of $\C[\mathcal{O}] = S/I(\mathcal{O})$. Geometrically, it is the coordinate ring of the flat limit of $\varepsilon\mathcal{O}$ as $\varepsilon \to 0$; one imagines a flat deformation of the points of $\mathcal{O}$ to a length-$|\mathcal{O}|$ scheme at the origin. In particular, we have
$$ \dim_\C S/ \mathrm{gr}I(\mathcal{O}) = |\mathcal{O}|$$
and the ring $S/\mathrm{gr}I(\mathcal O)$ is nonreduced whenever $|\mathcal{O}|>1$. If a finite subgroup $G \subseteq GL_d(\C)$ preserves $\mathcal{O}$, then we have the isomorphism of $G$-modules
$$ \C[\mathcal{O}] \cong_G S/\mathrm{gr}I(\mathcal{O})$$
so that the permutation representation on $\mathcal{O}$ is equipped with a grading.

For a partition $\mu \vdash n$, we recall the \emph{diagram} of $\mu$ is

$$ D_\mu = \{(c,h) \in \Z_{\geq 0}^2: 0 \leq h < \ell(\mu), 0 \leq c < \mu_{h+1}\} $$

\begin{figure}[htbp]
\centering
\begingroup

\newcommand{\OrbitCard}[3]{%
  \begin{scope}[shift={#1}]
    \path[
        draw=black!45,
        fill=white,
        line width=.45pt,
        rounded corners=2pt
    ] (0,0) rectangle (3.65,2.65);

    \node[anchor=north east,inner sep=5pt]
      at (3.65,2.65) {$#2$};

    \foreach \c/\h/\lab in {#3}{%
      \fill
        ({.42+.86*\c},{.40+.83*\h})
        circle (1.8pt);

      \node[above right=1.5pt,inner sep=0pt]
        at ({.42+.86*\c},{.40+.83*\h})
        {$\lab$};
    }
  \end{scope}%
}

\resizebox{\linewidth}{!}{%
\begin{tikzpicture}[
    >=Stealth,
    every node/.style={font=\small}
]


  \OrbitCard{(0,1.40)}{\rho_3}{%
    0/0/7,1/0/2,2/0/5,3/0/3,
    0/1/6,1/1/1,2/1/8,
    0/2/4}

  \OrbitCard{(.75,.70)}{\rho_2}{%
    0/0/2,1/0/8,2/0/4,3/0/6,
    0/1/1,1/1/5,2/1/7,
    0/2/3}

  \OrbitCard{(1.50,0)}{\rho_1}{%
    0/0/1,1/0/2,2/0/3,3/0/4,
    0/1/5,1/1/6,2/1/7,
    0/2/8}

  \node[font=\large] at (.40,.32) {$\cdots$};

  \node[align=center] at (2.58,-.52)
    {$\mathcal O_\mu\subseteq(\mathbb C^2)^8$};

  \node[align=center] at (2.58,-1.00)
    {all $8!$ labelings of $D_\mu$};


  \draw[
      ->,
      thick,
      decorate,
      decoration={
          snake,
          amplitude=1.2pt,
          segment length=8pt,
          pre length=2pt,
          post length=6pt
      }
  ] (5.50,1.98) -- (8.42,1.98);

  \node[align=center] at (6.96,2.60)
    {flat\\degeneration};

  \node at (6.96,1.48)
    {$\mu=(4,3,1)$};


  \node[
      cloud,
      cloud puffs=16,
      cloud puff arc=110,
      aspect=1.55,
      draw=black!50,
      fill=black!8,
      line width=.55pt,
      minimum width=3.35cm,
      minimum height=2.30cm,
      inner sep=0pt
  ] at (10.60,1.98) {};

  \fill (10.60,1.98) circle (2pt);

  \node[below right=3pt,inner sep=0pt]
    at (10.60,1.98) {$\mathbf 0$};

  \node[align=center] at (10.60,-.52)
    {$\operatorname{Spec}V_\mu\subseteq(\mathbb C^2)^8$};

  \node[align=center] at (10.60,-1.00)
    {supported at $\mathbf 0$, of length $8!$};

\end{tikzpicture}%
}

\endgroup
\caption{The orbit associated with $\mu=(4,3,1)$ degenerating to
$\operatorname{Spec}V_\mu$. Each card depicts a bijective labeling
$\rho:D_\mu\to[8]$ and represents one point $p_\rho$ of
$(\mathbb C^2)^8$; the stack stands for all $8!$ labelings.
The special fiber is supported at the origin of $(\mathbb C^2)^8$
and has length $8!$; the cloud denotes the ``scheme-y" point.}
\label{fig:orbit-labelings-degeneration}
\end{figure}
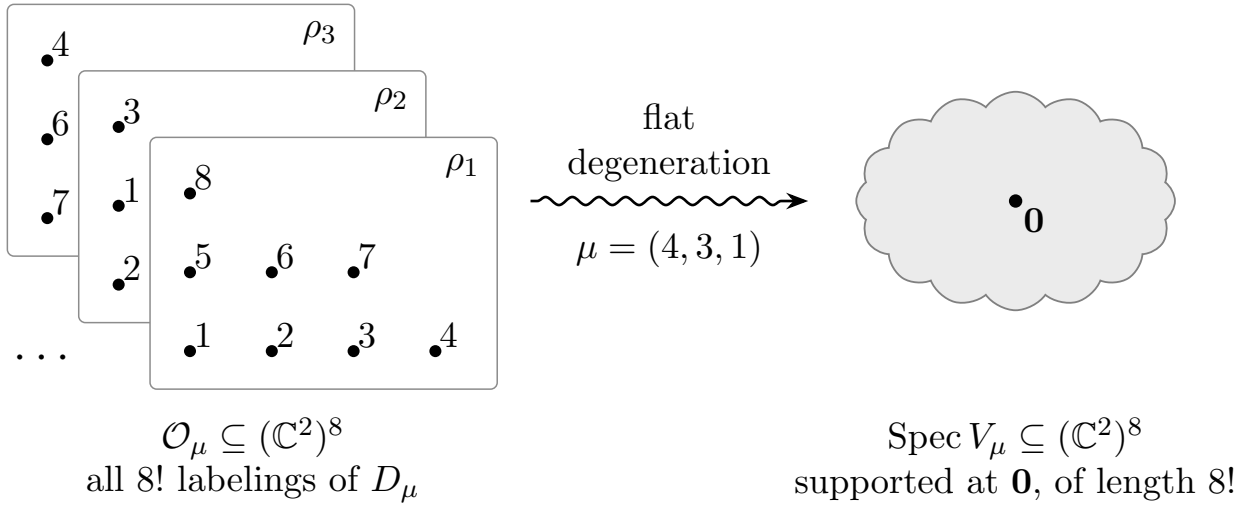

Write the coordinates of the points in $D_\mu$ in some order as:

$$ (a_1,b_1),\dots,(a_n,b_n)$$
and define
$$ \mathbf{p}_\mu := ( (a_1,b_1),\dots,(a_n,b_n)) \in (\C^2)^n$$

We define the \emph{orbit associated with $\mu$} or the \emph{$\mu$-diagram orbit} to be
$$ \mathcal{O}_\mu := \symm_n \cdot \mathbf{p}_\mu $$
where $\symm_n$ simply permutes the points. One sees that we may identify $(\C^2)^n$ with $\C^{2n}$, and this action becomes the usual diagonal action. Its coordinate ring is given by

$$ \mathcal{R}_\mu := \C[\mathcal{O}_\mu] = R/I(\mathcal{O}_\mu), \qquad R := \C[\xn,\yn]$$

where $I(\mathcal{O}_\mu)$ is the ideal of functions vanishing on the orbit $\mathcal{O}_\mu$.

\FloatBarrier
\section{Orbit-harmonics and the top-degree basis}\label{sec:orbit-newton}

In this section we prove Theorem~\ref{thma:intro}. We call the resulting monomial basis the \emph{Haglund basis}.

\subsection{Bigradings and the orbit harmonics ring}

Recall that $R = \C[\xn,\yn]$. Let

$$ F_a^\xv R := \{f \in R: \deg_\xv f \leq a\},\qquad F_a^\yv R= \{f \in R:\deg_\yv f \leq a\}$$

We may equip the orbit harmonics quotient $\mathcal{R}_\mu$ with these two filtrations by studying their images under the canonical quotient maps:

$$F_a^\xv \mathcal{R}_\mu := \mathrm{im}( F_a^\xv R \to \mathcal{R}_\mu),\qquad F_a^\yv \mathcal{R}_\mu := \mathrm{im}( F_a^\yv R \to \mathcal{R}_\mu)$$

with $F^\xv_{-1} = F^\yv_{-1} = 0$. If $V$ is a vector space with a filtration by degree $F$, we define the \emph{associated graded} to be
$$ \mathrm{gr}_FV:= \bigoplus_{d \geq 0} F_dV/F_{d-1}V$$

Concretely, we set

$$ \mathrm{gr}_\xv := \bigoplus_{d \geq 0} F^\xv_d / F^\xv_{d-1}, \qquad \mathrm{gr}_\yv := \bigoplus_{d \geq 0} F^\yv_d/F^\yv_{d-1}$$

We also have the associated bigraded algebra 

$$ \mathrm{bigr} \mathcal{R}_\mu := \bigoplus_{a,b\geq0} \frac{F^\xv_a \mathcal{R}_\mu \cap F^\yv_b \mathcal{R}_\mu}{(F^\xv_{a-1}\mathcal{R}_\mu \cap F^\yv_b\mathcal{R}_\mu) + (F^\xv_{a}\mathcal{R}_\mu \cap F^\yv_{b-1}\mathcal{R}_\mu)}$$

We may take iterated associated gradeds, where the second filtration is the polynomial-degree filtration on the first associated-graded quotient. These structures are compatible for the diagram orbit
$\mathcal R_\mu$, as stated in the following identification:

\begin{thm}\cite{garsia1996natural,haiman2001hilbert}\label{thm:bigr}
    There are natural $\symm_n$-equivariant isomorphisms of bigraded algebras
    $$ \mathrm{bigr} \mathcal{R}_\mu \cong \mathrm{gr}_\xv \mathrm{gr}_\yv \mathcal{R}_\mu \cong \mathrm{gr}_\yv \mathrm{gr}_\xv \mathcal{R}_\mu \cong V_\mu$$
\end{thm}

\begin{proof}
    Garsia and Haiman identify these constructions assuming the $n!$ conjecture \cite[Theorem~3.2, Equation~(3.29)]{garsia1996natural}. Haiman later establishes the $n!$ theorem \cite[Theorem~3.2]{haiman2001hilbert}.
\end{proof}

\begin{rem}
    As discussed by Garsia and Haiman \cite[Remark~3.1 and Theorem~3.2]{garsia1996natural}, the compatibility of the ``bigr" construction with the polynomial-degree filtrations and with $V_\mu$ is special to these partition-diagram orbits; this identification is a remarkable feature of the Hilbert scheme of points in $\C^2$. For $\C^3$, for instance, the analogous identification and equality of dimensions fail in general \cite[Section~5.2]{haiman2001hilbert}. Our construction heavily relies on this compatibility.
\end{rem}

For $0\le c<\mu_1$, let
\[
    h_c:=\mu'_{c+1}
\]
be the number of boxes in column $c$, and define
\begin{equation}\label{eq:top-y-quotient}
    N:=n(\mu)
      =\sum_{c=0}^{\mu_1-1}\binom{h_c}{2}, \qquad \mathcal Q_\mu
      :=F_N^\yv\mathcal{R}_\mu/
        F_{N-1}^\yv\mathcal{R}_\mu.
\end{equation}
This is a filtered vector space, with $\xv$-filtration
\begin{equation}\label{eq:top-x-filtration}
    F_d^\xv\mathcal Q_\mu
    :=
    \frac{
      (F_d^\xv\mathcal{R}_\mu\cap F_N^\yv\mathcal{R}_\mu)
        +F_{N-1}^\yv\mathcal{R}_\mu
    }{
      F_{N-1}^\yv\mathcal{R}_\mu
    }.
\end{equation}
where we set $F_{-1}^\xv\mathcal Q_\mu=0$. We will need the following corollary:

\begin{cor}\label{cor:top-filtration-layer}
There is a natural graded isomorphism
\[
    \operatorname{gr}_{\xv}\mathcal Q_\mu
       \cong V_\mu^{\mathrm{top}}.
\]
Moreover,
\begin{equation}\label{eq:top-filtered-hilbert}
    \sum_{d\ge0}
       \dim_\C\bigl(
          \operatorname{gr}_d^\xv\mathcal Q_\mu
       \bigr)q^d
    =
    \sum_{\sigma\in\CST_\mu}
       q^{|\invvec(\sigma)|}.
\end{equation}
Consequently,
\[
    \dim_\C F_d^\xv\mathcal Q_\mu
      =
      \#\{\sigma\in\CST_\mu:
                    |\invvec(\sigma)|\le d\},
\]
and
\[
    \dim_\C\mathcal Q_\mu
      =
      |\CST_\mu|
      =
      \binom{n}{\mu'}=\frac{n!}{\prod_{c=0}^{\mu_1-1}h_c!}.
\]
\end{cor}

\begin{proof}
By the definition of the induced filtration,
\[
 \operatorname{gr}_d^\xv\mathcal Q_\mu
 \cong
 \frac{F_d^\xv\mathcal{R}_\mu\cap F_N^\yv\mathcal{R}_\mu}
 {(F_{d-1}^\xv\mathcal{R}_\mu\cap F_N^\yv\mathcal{R}_\mu)
 +(F_d^\xv\mathcal{R}_\mu\cap F_{N-1}^\yv\mathcal{R}_\mu)}.
\]
This is the bidegree $(d,N)$ component of
$\mathrm{bigr}\,\mathcal{R}_\mu$, so Theorem~\ref{thm:bigr}
identifies it with $V_\mu^{(d,N)}$.
Haiman's graded-character theorem~\cite[Theorem~3.2]{haiman2001hilbert}
and the Haglund--Haiman--Loehr formula~\cite[Theorem~2.2]{haglund2005combinatorial}
give
\[
 \Hilb_q(V_\mu^{\mathrm{top}})
   =[m_{1^n}]\bigl([t^N]\Ht_\mu(X;q,t)\bigr)
   =\sum_{\sigma\in\CST_\mu}q^{|\invvec(\sigma)|},
\]
where $[m_{1^n}]$ denotes the $1^n$-coefficient in the monomial
symmetric function expansion. Summing the graded dimensions gives the filtration dimensions and the total dimension. Finally, assigning the labels to columns of sizes $h_c$ and then ordering each column increasingly gives $|\CST_\mu|=n!/\prod_c h_c!$.
\end{proof}

Finally, we record a notion of ``compatibility" of bases with filtrations that we will need. 

\begin{defn}
\label{def:filtered-basis}
Let $W$ be a finite-dimensional vector space with an
increasing exhaustive filtration
\[
    0=F_{-1}W\subseteq F_0W\subseteq F_1W
      \subseteq\cdots\subseteq W.
\]
A basis $\{b_i\}_{i\in I}$, with assigned degrees
$d_i\ge0$, is \emph{adapted to the filtration},
or a \emph{filtered basis}, if for every $d$ the set
\[
    \{b_i:d_i\le d\}
\]
is a basis of $F_dW$.
\end{defn}

In particular, if $v\in F_dW$ and
\[
    v=\sum_{i\in I}a_i b_i,
\]
then $a_i=0$ whenever $d_i>d$.
To see this, note the filtered-basis property gives an expansion
using only the $b_i$ with $d_i\le d$, and uniqueness of the expansion in the full basis completes the proof.

Another corollary is that, for every $d$, the set
\[
    \{b_i+F_{d-1}W:d_i=d\}
\]
forms a basis of $F_dW/F_{d-1}W$.
Therefore a filtered basis gives a homogeneous basis of the associated graded vector space.

For a polynomial $P$, let $\operatorname{in}_{\xv}(P)$
denote its highest $\xv$-homogeneous component.
For $v\in F_d^\xv\mathcal Q_\mu$, its degree-$d$ class is
\[
    [v]
       :=v+F_{d-1}^\xv\mathcal Q_\mu.
\]
Under the natural identification
$\operatorname{gr}_{\xv}\mathcal Q_\mu\cong V_\mu^\mathrm{top}$,
the degree-$d$ class of a polynomial homogeneous of
$\yv$-degree $N$ is given by its
$\xv$-homogeneous component of degree $d$.

\subsection{Column antisymmetrization}

We now define a family of functionals that detect the $\yv$-component of a monomial.

\begin{defn}
    Let $\sigma:D_\mu\to[n]$ be a bijective filling (not necessarily column-strict) of $\mu$. If $(c_\sigma(i),h_\sigma(i))$ denotes the zero-based coordinates of the box labeled $i$, we define
    $$ \mathbf{p}_\sigma := ( (c_\sigma(1),h_\sigma(1)),\dots,(c_\sigma(n),h_\sigma(n))) \in (\C^2)^n$$
    These points identify the bijective fillings with the elements of $\mathcal O_\mu$. For a function $f \in \mathcal{R}_\mu$, we define
    $$ f(\mathbf{p}_\sigma) = f(c_\sigma(1),\dots,c_\sigma(n);h_\sigma(1),\dots,h_\sigma(n))$$
    to be the function evaluated at the tuple $\mathbf{p}_\sigma$.
\end{defn}

\begin{example}
    For the following filling $\sigma$ of shape $(2,2)$, we have
    $$
    \sigma = 
    \begin{ytableau}
        3 & 4 \\
        1 & 2
    \end{ytableau},
    \qquad
    \mathbf{p}_\sigma = ( (0,0),(1,0),(0,1),(1,1))
    $$
    Evaluation of a function $f \in \mathcal{R}_{(2,2)}$ would then denote evaluation at
    $$ (x_1,x_2,x_3,x_4) = (0,1,0,1), \qquad (y_1,y_2,y_3,y_4) = (0,0,1,1)$$
    For instance,
    $$ (y_3y_4)(\mathbf{p}_\sigma) = 1 \cdot 1 = 1$$
\end{example}

Fix $\tau\in\CST_\mu$.  Two bijective fillings are \emph{column-equivalent}
if each geometric column contains the same set of labels in both fillings.
Let $[\tau]_{\mathrm{col}}$ be the column-equivalence class of $\tau$.  Every
$\rho\in[\tau]_{\mathrm{col}}$ is obtained by independently permuting the
labels in each column of $\tau$.  Let $\epsilon_\tau(\rho)$ be the product of
the signs of these column permutations, and define the column-alternating
functional
\[
    \mathcal A_\tau(f)
      :=\sum_{\rho\in[\tau]_{\mathrm{col}}}
             \epsilon_\tau(\rho)f(\mathbf{p}_\rho),
    \qquad f\in\mathcal{R}_\mu.
\]

\begin{example}
    Let $\mu=(2,2)$ and consider the filling
\[
    \sigma=
    \begin{ytableau}
        3&4\\
        1&2
    \end{ytableau}.
\]
Its column-equivalence class consists of the four fillings
\[
    \sigma=
    \begin{ytableau}
        3&4\\
        1&2
    \end{ytableau},
    \qquad
    \sigma_L=
    \begin{ytableau}
        1&4\\
        3&2
    \end{ytableau},
    \qquad
    \sigma_R=
    \begin{ytableau}
        3&2\\
        1&4
    \end{ytableau},
    \qquad
    \sigma_{LR}=
    \begin{ytableau}
        1&2\\
        3&4
    \end{ytableau}.
\]
Here $L$ and $R$ indicate swapping the entries in the left
and right columns, respectively. Their signs are
$+1,-1,-1,+1$, so
\[
    \mathcal A_\sigma(f)
       =f(\mathbf{p}_\sigma)-f(\mathbf{p}_{\sigma_L})
        -f(\mathbf{p}_{\sigma_R})+f(\mathbf{p}_{\sigma_{LR}}).
\]
For instance, we have the following evaluations
$$\mathcal A_\sigma(y_3y_4)=1-0-0+0=1, \qquad  \mathcal A_\sigma(y_3)=1-0-1+0=0.$$
\end{example}

For a column $c$, let its height be $h_c$, and write its
labels in $\tau$ from bottom to top as
\[
    i_{c,0}<i_{c,1}<\cdots<i_{c,h_c-1},
\]
where the inequalities follow from column strictness.

\begin{lem}[column factorization]\label{lem:column-alt-factorization}
For exponent vectors $\eta,\beta\in\Z_{\ge0}^n$,
\begin{equation}\label{eq:column-alt-factorization}
 \mathcal A_\tau(\xv^\eta \yv^\beta)
 =(\xv^\eta)(\mathbf{p}_\tau)\mathcal{A}_\tau(\yv^\beta) = \left(\prod_{a=1}^n c_\tau(a)^{\eta_a}\right)
   \prod_{c=0}^{\mu_1-1}
   \det\!\left(r^{\beta_{i_{c,s}}}\right)_{
          0\le s,r<h_c},
\end{equation}
where $0^0=1$ and $(\xv^\eta)(\mathbf{p}_\tau)$ denotes evaluation at $\mathbf{p}_\tau$.
In particular, $\mathcal A_\tau$ annihilates every function represented by a
polynomial of $\yv$-degree strictly smaller than $N$ and therefore descends to a
linear functional on $\mathcal Q_\mu$.
\end{lem}

\begin{proof}
Within a column-equivalence class, the $\xv$-coordinate of each label is fixed,
so the $\xv$-factor may be factored out.  The signed sum over permutations in a
single column is the determinant in
Equation~\eqref{eq:column-alt-factorization}; since the choices in distinct columns
are independent, the sum can be written as the above product.

If one of these determinants is nonzero, then the exponents assigned to that
column are pairwise distinct.  This implies their sum is at least
$0+1+\cdots+(h_c-1)=\binom{h_c}{2}$.  If every column determinant were
nonzero, then $|\beta|\ge\sum_c\binom{h_c}{2}=N$.  Therefore every monomial of
$y$-degree less than $N$ is annihilated, and linearity proves the last
assertion.
\end{proof}

\begin{example}
    Consider the column-strict filling
\[
    \tau=
    \begin{ytableau}
        5 & \none\\
        3 & 4\\
        1 & 2
    \end{ytableau}, \qquad \mathbf{p}_\tau = ( (0,0),(1,0),(0,1),(1,1),(0,2))
\]
so that the columns (written in increasing order) are given by
$$ (i_{0,0},i_{0,1},i_{0,2}) = (1,3,5), \qquad (i_{1,0},i_{1,1}) = (2,4)$$
We calculate, as an example, $\mathcal{A}_\tau(x_2^2x_4y_1y_4y_5^2)$ using the determinant formula. We have that
$$ \eta = (0,2,0,1,0), \qquad \beta = (1,0,0,1,2)$$
so that the factors $\beta_{i_{c,s}}$ are given by
$$ (\beta_1,\beta_3,\beta_5) = (1,0,2), \qquad (\beta_2,\beta_4) = (0,1)$$
The $\xv$-factor has the constant value $(x_2^2x_4)(\mathbf{p}_\tau) = 1^2 \cdot 1 = 1$ since both $2,4$ are in the column labeled $1$ in $\tau$. The determinant formula then gives
\begin{align*}
\mathcal{A}_\tau(x_2^2x_4y_1y_4y_5^2) &= (1^2\cdot1) \det\begin{pmatrix}
        0^{\beta_1}&1^{\beta_1}&2^{\beta_1}\\
        0^{\beta_3}&1^{\beta_3}&2^{\beta_3}\\
        0^{\beta_5}&1^{\beta_5}&2^{\beta_5}
    \end{pmatrix}
    \det
    \begin{pmatrix}
        0^{\beta_2}&1^{\beta_2}\\
        0^{\beta_4}&1^{\beta_4}
    \end{pmatrix} \\
    &= \det\begin{pmatrix}
        0^1&1^1&2^1\\
        0^0&1^0&2^0\\
        0^2&1^2&2^2
    \end{pmatrix}
    \det\begin{pmatrix}
        0^0&1^0\\
        0^1&1^1
    \end{pmatrix} = -2
\end{align*}
where $0^0 = 1$.
\end{example}

An important fact is that, on monomials of total $\yv$-degree $N$, the functional $\mathcal A_\tau$ is supported on $\yv$-words that are weakly greater than $\rowvec(\tau)$ in lexicographical order. This fact will play a key role in the triangularity argument. We prove the following lemma:

\begin{lem}[row-word triangularity]\label{lem:row-word-triangularity}
Suppose $|\beta|=N$ and
$\mathcal A_\tau(\xv^\eta \yv^\beta)\ne0$.  Then
\[
    \beta\ge_{\mathrm{lex}}\rowvec(\tau).
\]
If equality holds, the determinant product in
Equation~\eqref{eq:column-alt-factorization} is the fixed nonzero constant
$$\Upsilon_\mu := \prod_{c = 0}^{\mu_1-1}\prod_{r = 0}^{h_c-1} r!$$
which depends only on $\mu$. Here, $h_c$ is the height of column $c$.
\end{lem}

\begin{proof}
Nonvanishing and total degree $N$ force the exponents in every column to be
exactly the set $\{0,1,\ldots,h_c-1\}$.  Put
$\gamma=\rowvec(\tau)$, and suppose $a$ is the first label for which
$\beta_a\ne\gamma_a$.  If $a=i_{c,r}$, then the smaller labels
$i_{c,0},\ldots,i_{c,r-1}$ occur before $a$ and, by the choice of $a$, have
$\beta$-values $0,\ldots,r-1$.  Since the $\beta$-values in this column are
distinct, $\beta_a$ cannot be smaller than $r$.  It is not equal to $r$, so
\[
    \beta_a>r=\gamma_a.
\]
This proves the lexicographic inequality.  When $\beta=\gamma$, for each column $c$  Lemma \ref{lem:column-alt-factorization} implies each column
determinant becomes the Vandermonde
$$ \det(r^s)_{0\leq s,r<h_c} = \det
\begin{pmatrix}
    0^0 & 1^0 & \dots & (h_c-1)^0 \\
    0^1 & 1^1 & \dots & (h_c-1)^1 \\
    \vdots & \vdots & \ddots & \vdots \\
    0^{h_c-1} & 1^{h_c-1} & \dots & (h_c-1)^{h_c-1}
\end{pmatrix} = \prod_{0 \leq s < r < h_c}(r-s) = \prod_{r = 1}^{h_c-1} r!
$$
Multiplying across all columns gives
\[
    \mathcal{A}_\tau(\xv^\eta\yv^\beta) =(\xv^\eta)(\mathbf{p}_\tau)\bigg(\prod_{c = 0}^{\mu_1-1} \prod_{r = 0}^{h_c-1} r!\bigg) = (\xv^\eta)(\mathbf{p}_\tau)\Upsilon_\mu\qedhere
\]
\end{proof}

\begin{example}
    Let
\[
    \sigma=
    \begin{ytableau}
        3&4\\
        1&2
    \end{ytableau},
    \qquad
    \tau=
    \begin{ytableau}
        2&4\\
        1&3
    \end{ytableau},
\]
with row vectors
\[
    \operatorname{rowvec}(\sigma)=(0,0,1,1),
    \qquad
    \operatorname{rowvec}(\tau)=(0,1,0,1),
\]
so that $\operatorname{rowvec}(\sigma)
       <_{\mathrm{lex}}\operatorname{rowvec}(\tau)$. Indeed, we have

\begin{align*}
    \mathcal A_\tau(y_3y_4)
    &=
    \det\begin{pmatrix}
        1&1\\
        1&1
    \end{pmatrix}
    \det\begin{pmatrix}
        0&1\\
        0&1
    \end{pmatrix}
    =0\\
    \mathcal{A}_\sigma(y_3y_4) &=
    \det\begin{pmatrix}
        1&1\\
        0&1
    \end{pmatrix}
    \det\begin{pmatrix}
        1&1\\
        0&1
    \end{pmatrix}
    =1
\end{align*}
\end{example}

\subsection{Haglund insertion algorithm}

We first describe an insertion algorithm that reconstructs a
column-strict filling from its row vector and inversion tuple.

\begin{defn}[Haglund insertion]\label{def:insertion-algorithm}
Let $\alpha\in\mathbb Z_{\ge0}^n$, and let
$\beta\in\mathbb Z_{\ge0}^n$ have content
\[
    0^{\mu_1}1^{\mu_2}\cdots
    (\ell(\mu)-1)^{\mu_{\ell(\mu)}}.
\]
We set \[
    B_j(\beta)=\{a:\beta_a=j-1\}.
\]
Start with the empty diagram of shape $\mu$. Process the rows
from bottom to top and, within row $j$, process the labels of
$B_j(\beta)$ in increasing order.

\Needspace{7\baselineskip}
When processing a label $a\in B_j(\beta)$, a cell in row $j$
is \emph{eligible} if it is unoccupied and either:
\begin{enumerate}
    \item $j=1$; or
    \item $j>1$ and the entry directly below it is smaller than $a$.
\end{enumerate}
List the eligible columns in increasing order as
\[
    e_0<e_1<\cdots<e_{d-1}.
\]
If $\alpha_a<d$, place $a$ in column $e_{\alpha_a}$.
If $\alpha_a\ge d$, the algorithm fails and stops.

If every label is inserted, return the resulting filling.
\end{defn}

The integer $\alpha_a$ is the number of eligible cells skipped
when inserting $a$, not the column number of its chosen cell.
All column numbers and eligible-list indices begin at $0$.

For a fixed $\sigma\in\CST_\mu$, reconstruct $\sigma$ in the
same row-by-row order, placing each label in its prescribed
cell. Let
\[
    E_a^\sigma
       =\{e_{a,0}^\sigma<\cdots<
          e_{a,|E_a^\sigma|-1}^\sigma\}
\]
be the eligible-column list when $a$ is processed. Define
$\operatorname{code}_\sigma(a)$ to be the unique index satisfying
\[
    c_\sigma(a)
       =e_{a,\operatorname{code}_\sigma(a)}^\sigma.
\]
The prescribed cell is eligible because $\sigma$ is
column-strict.

\begin{example}\label{ex: haglund-insertion-54}
    Let $\mu = (5,4) \vdash 9$, and consider the following $\sigma \in \CST_{(5,4)}$:

    \[
    \sigma=
    \begin{ytableau}
        9 & 8 & 4 & 3 & \none\\
        5 & 6 & 2 & 1 & 7\\
        \none[\scriptstyle 0] &
        \none[\scriptstyle 1] &
        \none[\scriptstyle 2] &
        \none[\scriptstyle 3] &
        \none[\scriptstyle 4]
    \end{ytableau},
    \]
    where the labels below denote the column number. We will perform the insertion with the following vectors to recover $\sigma$:
    $$\alpha = (3,2,1,0,0,0,0,1,0), \qquad \beta = (0,0,1,1,0,0,0,1,1) $$
    We perform the insertion, with cyan denoting skipped eligible cells, and pink denoting the rest of the eligible cells:
    \par
\begingroup
\small
\setlength{\tabcolsep}{5pt}
\renewcommand{\arraystretch}{1.3}
\setlength{\LTpre}{6pt}
\setlength{\LTpost}{6pt}
\setlength{\LTleft}{\fill}
\setlength{\LTright}{\fill}
\begin{longtable}{@{}cccccccc@{}}
\toprule
Step $t$
& Label $a_t$
& Row $j$
& $E_{a_t}^\sigma$
& $|E_{a_t}^\sigma|$
& $\alpha_{a_t}$
& \shortstack{Chosen column\\
    $e_{a_t,\alpha_{a_t}}^\sigma$}
& Partial filling
\\
\midrule
\endfirsthead
\multicolumn{8}{@{}l}{\itshape Haglund insertion for $\mu=(5,4)$ (continued)}\\[3pt]
\toprule
Step $t$
& Label $a_t$
& Row $j$
& $E_{a_t}^\sigma$
& $|E_{a_t}^\sigma|$
& $\alpha_{a_t}$
& \shortstack{Chosen column\\
    $e_{a_t,\alpha_{a_t}}^\sigma$}
& Partial filling
\\
\midrule
\endhead
\multicolumn{8}{r@{}}{\footnotesize\itshape Continued on the next page}\\
\endfoot
\bottomrule
\endlastfoot
1&1&1&$\{0,1,2,3,4\}$&5&3&3&
\begin{ytableau}
    \blankcell & \blankcell & \blankcell & \blankcell \\
    *(cyan) & *(cyan) & *(cyan) & *(pink)1 & *(pink)
\end{ytableau}
\\
\midrule
2&2&1&$\{0,1,2,4\}$&4&2&2&
\begin{ytableau}
    \blankcell & \blankcell & \blankcell & \blankcell \\
    *(cyan) & *(cyan) & *(pink)2 & 1 & *(pink)
\end{ytableau}
\\
\midrule
3&5&1&$\{0,1,4\}$&3&0&0&
\begin{ytableau}
    \blankcell & \blankcell & \blankcell & \blankcell \\
    *(pink)5 & *(pink) & 2 & 1 & *(pink)
\end{ytableau}
\\
\midrule
4&6&1&$\{1,4\}$&2&0&1&
\begin{ytableau}
    \blankcell & \blankcell & \blankcell & \blankcell \\
    5 & *(pink)6 & 2 & 1 & *(pink)
\end{ytableau}
\\
\midrule
5&7&1&$\{4\}$&1&0&4&
\begin{ytableau}
    \blankcell & \blankcell & \blankcell & \blankcell \\
    5 & 6 & 2 & 1 & *(pink)7
\end{ytableau}
\\
\midrule
6&3&2&$\{2,3\}$&2&1&3&
\begin{ytableau}
    \blankcell & \blankcell & *(cyan) & *(pink)3 \\
    5 & 6 & 2 & 1 & 7
\end{ytableau}
\\
\midrule
7&4&2&$\{2\}$&1&0&2&
\begin{ytableau}
    \blankcell & \blankcell & *(pink)4 & 3 \\
    5 & 6 & 2 & 1 & 7
\end{ytableau}
\\
\midrule
8&8&2&$\{0,1\}$&2&1&1&
\begin{ytableau}
    *(cyan) & *(pink)8 & 4 & 3 \\
    5 & 6 & 2 & 1 & 7
\end{ytableau}\\
\midrule
9&9&2&$\{0\}$&1&0&0&
\begin{ytableau}
    *(pink)9 & 8 & 4 & 3 \\
    5 & 6 & 2 & 1 & 7
\end{ytableau}
\\
\end{longtable}
\endgroup
The algorithm returns $\sigma$. Its code, written in
insertion order, is
\[
    (3,2,0,0,0,1,0,1,0),
\]
whereas the same code, written in label order, is
\[
    \alpha^\sigma
       =\operatorname{invvec}(\sigma)
       =(3,2,1,0,0,0,0,1,0).
\]
The nonzero entries can also be read directly from the filling:
\begin{center}
\begin{tabular}{@{}clc@{}}
\toprule
Label $a$
& Inversions attributed to $a$
& $\alpha_a^\sigma$\\
\midrule
1
& $(-\infty,1,5),(-\infty,1,6),(-\infty,1,2)$
& 3\\
2
& $(-\infty,2,5),(-\infty,2,6)$
& 2\\
3
& $(2,3,4)$
& 1\\
8
& $(5,8,9)$
& 1\\
\bottomrule
\end{tabular}
\end{center}
All other labels contribute no inversions.
\end{example}

\Needspace{9\baselineskip}
\subsection{Insertion codes and Newton polynomials}

As before, let $\beta$ be a row word and write
\[
    B_j(\beta)=\{a:\beta_a=j-1\}.
\]
If $\sigma\in\CST_\mu$, perform Haglund insertion as in Definition \ref{def:insertion-algorithm}; at the moment when $a$ is inserted,
let
\[
    E_a^\sigma=
      \{e_{a,0}^\sigma<e_{a,1}^\sigma<\cdots<
                     e_{a,|E_a^\sigma|-1}^\sigma\}
\]
be the eligible columns. Let $\operatorname{code}_\sigma(a)$ be the index of the column occupied
by $a$ in this list.

\begin{lem}
\label{lem:code-equals-invvec}
For every $\sigma\in\CST_\mu$ and every label $a$,
\[
    \operatorname{code}_\sigma(a)=\invvec(\sigma)_a.
\]
Consequently the insertion algorithm applied to
$(\invvec(\sigma),\rowvec(\sigma))$ reconstructs $\sigma$.
\end{lem}

\begin{proof}
When $a$ is inserted, each skipped eligible cell to the left of its chosen cell
will later receive a larger entry $c$ in the same row.
In the bottom row each such cell gives the inversion pair
$(-\infty,a,c)$.  In a higher row its lower neighbor is an entry $u<a$, so it
gives the inversion triple $(u,a,c)$.  Conversely, every inversion whose
middle entry is $a$ arises from exactly one such skipped eligible column, so that the number of skipped eligible columns is
$\invvec(\sigma)_a$, as claimed.
\end{proof}

\begin{defn}\label{def: newton-poly}
    For $\sigma\in\CST_\mu$, put $\alpha^\sigma=\invvec(\sigma)$ and define its \emph{Newton polynomial} by
    \begin{equation}\label{eq:newton-polynomial}
        \mathcal N_\sigma(x)
          :=\prod_{a=1}^n
              \prod_{0\le q<\alpha_a^\sigma}
                  \bigl(x_a-e_{a,q}^\sigma\bigr).
    \end{equation}
    Its $\xv$-degree is $|\alpha^\sigma|$, and its highest $\xv$-homogeneous
    component is $\xv^{\alpha^\sigma} = \xv^{\invvec(\sigma)}$.
\end{defn}

\begin{example}
    Let $\mu = (5,4)$ as before, and $\sigma$ be our running example
    \[
    \sigma=
    \begin{ytableau}
        9 & 8 & 4 & 3 & \none\\
        5 & 6 & 2 & 1 & 7
    \end{ytableau},
\]
with inversion vector and row vector
\[
        \alpha^\sigma=(3,2,1,0,0,0,0,1,0),\qquad
        \beta^\sigma=(0,0,1,1,0,0,0,1,1).
\]
Write
\[
    F_a^\sigma(x_a)
       :=\prod_{0\le q<\alpha_a^\sigma}
                   (x_a-e_{a,q}^\sigma),
    \qquad
    \mathcal N_\sigma(x)
       :=\prod_{a=1}^{9}F_a^\sigma(x_a).
\]
The following table keeps track of the columns skipped during the insertion process, as in Example \ref{ex: haglund-insertion-54}:

\begin{center}
\small
\setlength{\tabcolsep}{5pt}
\begin{tabular}{@{}cccl@{}}
\toprule
$a$
& $E_a^\sigma$
& \shortstack{Skipped\\columns}
& $F_a^\sigma(x_a)$\\
\midrule
1
&$\{0,1,2,3,4\}$&$0,1,2$
&$x_1(x_1-1)(x_1-2)$\\
2
&$\{0,1,2,4\}$&$0,1$
&$x_2(x_2-1)$\\
5
&$\{0,1,4\}$&$\varnothing$
&$1$\\
6
&$\{1,4\}$&$\varnothing$
&$1$\\
7
&$\{4\}$&$\varnothing$
&$1$\\
\midrule
3
&$\{2,3\}$&$2$
&$x_3-2$\\
4
&$\{2\}$&$\varnothing$
&$1$\\
8
&$\{0,1\}$&$0$
&$x_8$\\
9
&$\{0\}$&$\varnothing$
&$1$\\
\bottomrule
\end{tabular}
\end{center}
to obtain the Newton polynomial
\[
    \mathcal N_\sigma(x)
      =x_1(x_1-1)(x_1-2)\,
       x_2(x_2-1)\,
       (x_3-2)\,x_8.
\]
Note the highest degree component is $x_1^3x_2^2x_3x_8 = \xv^{\invvec(\sigma)}$.
\end{example}

\begin{rem}
    The name \emph{Newton polynomial} was chosen due to the resemblance to polynomials that arise in multivariate Newton interpolation; see \cite[Equation~(3)]{neidinger2019newton}. In Neidinger's notation, given a multi-index $\lambda$ and a corresponding grid node $\xv(\lambda)$, with the coordinate nodes $x_m(i)$ pairwise distinct for each $m$, the \emph{classic Newton polynomial} is
    $$ q_\lambda(\xv) = \prod_{m = 1}^n \prod_{i = 0}^{\lambda_m-1} (x_m - x_m(i))$$
\end{rem}

Fix a row word $\beta$ and write
\[
    \CST_\mu(\beta)
       :=\{\sigma\in\CST_\mu:\rowvec(\sigma)=\beta\}.
\]
for the collection of column strict fillings with row word $\beta$. We order this set lexicographically by its insertion-code. In other words, sort lexicographically by
$$ (\invvec(\sigma)_{a_1},\dots,\invvec(\sigma)_{a_n})$$
where $a_1,\dots,a_n$ is the order in which the entries were inserted (the $a$ column in Example \ref{ex: haglund-insertion-54}).

\begin{lem}[Newton triangularity]\label{lem:newton-triangularity}
The matrix
\[
    \left(\mathcal N_\sigma(\mathbf{p}_\tau)\right)_{
       \tau,\sigma\in\CST_\mu(\beta)}
\]
is triangular with nonzero diagonal. In particular, the matrix is invertible.
\end{lem}

\begin{proof}
Let $a$ be the first insertion step at which $\sigma$ and $\tau$ make
different choices.  Before this step their partial fillings, and hence their
eligible-column lists, coincide.  If
$\operatorname{code}_\tau(a)<\operatorname{code}_\sigma(a)$, then
\[
    c_\tau(a)=e_{a,\operatorname{code}_\tau(a)}^\sigma
\]
is one of the roots appearing in the $a$-factor of
$\mathcal N_\sigma$, so
$\mathcal N_\sigma(\mathbf p_\tau)=0$.  On the diagonal,
\[
    \mathcal N_\sigma(\mathbf{p}_\sigma)
      =\prod_a\prod_{q<\alpha_a^\sigma}
         \bigl(e_{a,\alpha_a^\sigma}^\sigma-e_{a,q}^\sigma\bigr) \neq 0.
\]
This completes the proof.
\end{proof}

\subsection{Proof of the basis theorem}

We now prove the main result of this paper.

For $\sigma\in\CST_\mu$, define
\begin{equation}\label{eq:interpolation-polynomial}
    \mathcal G_\sigma(\xv,\yv)
      :=\mathcal N_\sigma(\xv)\yv^{\rowvec(\sigma)}.
\end{equation}
Consider the square matrix
\[
    \mathcal M_\mu
       :=\left(\mathcal A_\tau(\mathcal G_\sigma)\right)_
             {\tau,\sigma\in\CST_\mu},
\]
where the rows and columns are ordered first by the lexicographic order of their row
words (the $\yv$-exponent) and then, within a fixed row word, by insertion code (the $\xv$-exponents).

\begin{prop}\label{prop:interpolation-triangularity}
The matrix $\mathcal M_\mu$ is invertible.
\end{prop}

\begin{proof}
Lemma \ref{lem:column-alt-factorization} together with linearity of $\mathcal{A}_\tau$ gives
$$ \mathcal{A}_\tau(\mathcal{G}_\sigma) = \mathcal{A}_\tau(\mathcal{N}_\sigma(\xv)\yv^{\rowvec(\sigma)}) =\mathcal{N}_\sigma(\mathbf{p}_\tau) \mathcal{A}_\tau(\yv^{\rowvec(\sigma)})$$
where $\mathcal{N}_\sigma(\mathbf{p}_\tau)$ is the Newton polynomial evaluated at $\mathbf{p}_\tau$. This factorization is legal since $\mathcal{N}_\sigma(\xv)$ depends only on the $\xv$-variables.
By Lemma~\ref{lem:row-word-triangularity}, $\mathcal{A}_\tau(\yv^{\rowvec(\sigma)})$ can be nonzero only if $\rowvec(\sigma)\ge_{\mathrm{lex}}\rowvec(\tau)$, so that $\mathcal M_\mu$ is block triangular with blocks indexed by row words.
On a diagonal block, Lemma~\ref{lem:row-word-triangularity} gives
\[
    \mathcal A_\tau(\mathcal G_\sigma)
       =\Upsilon_\mu\mathcal N_\sigma(\mathbf{p}_\tau).
\]
The diagonal block is therefore a nonzero scalar multiple of the invertible matrix from Lemma~\ref{lem:newton-triangularity}.  Therefore $\mathcal M_\mu$ is block upper triangular with invertible diagonal blocks, and thus is invertible. The claim is proven.
\end{proof}

\begin{thm}[Haglund basis; Theorem~\ref{thma:intro}]\label{thm:ghbasis}
The monomials
\[
    \Basis_\mu^{\mathrm{top}}
      =\left\{\xv^{\invvec(\sigma)}\yv^{\rowvec(\sigma)}:
                 \sigma\in\CST_\mu\right\}
\]
form a basis of $V_\mu^{\mathrm{top}}$.
\end{thm}

\begin{proof}
Suppose
\[
    \sum_{\sigma\in\CST_\mu}a_\sigma\mathcal G_\sigma
       \in F^\yv_{N-1}\mathcal{R}_\mu,
\]
where $a_\sigma\in\C$ for each $\sigma\in\CST_\mu$.
Every functional $\mathcal{A}_\tau$ annihilates the lower filtration
piece by
Lemma~\ref{lem:column-alt-factorization}.  Applying all
$\mathcal A_\tau$ gives
\[
    \mathcal M_\mu(a_\sigma)_\sigma=0.
\]
Since $\mathcal{M}_\mu$ is invertible by Proposition~\ref{prop:interpolation-triangularity}, this implies
$a_\sigma=0$ for every $\sigma$.  Hence the images of the
$\mathcal G_\sigma$ are linearly independent in $\mathcal Q_\mu$.

Finally, by Corollary \ref{cor:top-filtration-layer}
\[
    \dim\mathcal Q_\mu= \binom{n}{\mu'}=|\CST_\mu|,
\]
so that the images of the $\mathcal G_\sigma$ form a basis of
$\mathcal Q_\mu$.

Let $F^\xv_d\mathcal Q_\mu$ be the induced $\xv$-degree filtration.  The
polynomial $\mathcal G_\sigma$ lies in filtration degree
$|\invvec(\sigma)|$.  Equation~\eqref{eq:top-filtered-hilbert} gives
\[
 \dim F^\xv_d\mathcal Q_\mu
   =\#\{\sigma\in\CST_\mu:
                  |\invvec(\sigma)|\le d\}.
\]
The corresponding subset of the already independent
$\mathcal G_\sigma$ therefore forms a basis of $F^\xv_d\mathcal Q_\mu$ for
every $d$.  Hence the $\mathcal G_\sigma$ form a filtered basis (cf. Definition \ref{def:filtered-basis}).

Passing to the $\xv$-associated graded, their initial forms are
\[
    \operatorname{in}_\xv(\mathcal G_\sigma)
       =\xv^{\invvec(\sigma)}\yv^{\rowvec(\sigma)}.
\]
Theorem \ref{thm:bigr} identifies these initial
forms with classes in $V_\mu^{\tp}$, and a filtered basis passes to a
homogeneous basis of the associated graded.  The proof is complete.
\end{proof}


\section{\texorpdfstring{$q$}{q}-Whittaker polynomials and the fermionic formula}
\label{sec:fermionic}

Let $N=n(\mu)$, and let $q$ record $\xv$-degree.
We have the $q$-Whittaker polynomial
\begin{equation}\label{eq:q-whittaker-definition}
    W_\mu(X;q):=P_\mu(X;q,0)
       =[t^N]\Ht_\mu(X;q,t);
\end{equation}
see \cite[Section~3.1, Equation~(3.1)]{bhattacharya2024whittaker} and \cite{bergeron2020whittaker}.
In this section, we define a \emph{flag inversion code} based on a construction of \cite{bhattacharya2024whittaker}, and use it to construct the flag basis of $V_\mu^{\tp}$. The advantage of the flag basis is that for a fixed $\yv$-monomial, say $\yv^\beta$, the $\xv$-exponents can be described by a family of simple bounds determined by the row word $\beta$. Both the flag basis and the Haglund basis
of Theorem~\ref{thm:ghbasis} behave nicely with respect to lexicographic
filtration of the $\yv$-monomials, and we use this fact to show that both bases witness the ``fermionic formula" of the $m_{1^n}$-coefficient of $W_\mu(X;q)$.

\subsection{Row words, partition flags, and inversion codes}

Let $\SYT(\mu)$ denote the set of standard Young tableaux of shape $\mu$ in French notation (so that rows increase left to right, columns increase bottom to top). 

\begin{lem}[{\cite[Lemma~5.1]{bhattacharya2024whittaker}}]
    Let $\sigma \in \CST_\mu$. If we sort the rows of $\sigma$, the resulting tableau $\operatorname{rowsort}(\sigma)$ is a standard Young tableau, i.e. $\operatorname{rowsort}(\sigma)\in\SYT(\mu)$.
\end{lem}

\begin{defn}
    Let $\sigma \in \CST_\mu$. We define a map $\rowsort: \CST_\mu \to \SYT(\mu)$ that simply sorts the rows. 
\end{defn}

This map has the correct codomain by the previous lemma.

\begin{defn}
    Let $\beta = (\beta_1,\dots,\beta_n) \in \mathbb{Z}_{\geq 0}^n$ be a word with content $\mu$, i.e.
    $$ \operatorname{content}(\beta) = 0^{\mu_1}1^{\mu_2}\dots(\ell-1)^{\mu_\ell}$$
    where $\ell$ is the length of $\mu$. A \emph{prefix} of $\beta$ is a truncation $(\beta_1,\dots,\beta_i)$ for some $i \leq n$. A \emph{lattice word} is a word $\beta$ such that every prefix of $\beta$ contains at least as many $r-1$'s as $r$'s, for every $r \geq 1$. Note that this can only occur when $\mu$ is a partition. We denote the set of lattice words with content $\mu$ by $\Lat_\mu$.
\end{defn}

It is a well-known combinatorial fact that standard Young tableaux of shape $\mu$ are in bijection with lattice words $\Lat_\mu$. 

\begin{defn}
    Let $\beta \in \Lat_\mu$. We denote the set of column-strict fillings with a fixed row word $\beta$ by
    $$\CST_\mu(\beta) := \{\sigma \in \CST_\mu:\rowvec(\sigma) = \beta\}$$
\end{defn}

\begin{defn}
    A \emph{complete flag of partitions} or \emph{complete partition flag} is an increasing sequence of partitions

    $$ \lambda^\bullet = (\varnothing = \lambda^{(0)} \subseteq \lambda^{(1)} \subseteq \dots \subseteq \lambda^{(n)}  = \mu)$$
    where $\lambda^{(i)} \vdash i$. If $\lambda^{(n)} = \mu$, we say that $\lambda^\bullet$ is a \emph{complete flag of} $\mu$.
\end{defn}

These flags can be viewed as a saturated chain in Young's lattice of partitions. We remark that our complete flags are a subclass of \emph{Gelfand--Tsetlin patterns}, which often appear in the $q$-Whittaker context; see \cite[Remark~2.2]{molev2006gelfand}. 

\begin{defn}
    Let $\beta \in \Lat_\mu$. We define the integers
    $$ \lambda^{(a)}(\beta)_r := \#\{i \leq a : \beta_i = r-1\}, \qquad 0 \leq a \leq n$$
    We define the $a$-th \emph{prefix shape} to be the partition
    $$ \lambda^{(a)}(\beta) = (\lambda^{(a)}(\beta)_1,\dots,\lambda^{(a)}(\beta)_a)$$
    and the $a$-th \emph{gap} of $\beta$ to be the difference
    $$ g_a(\beta) := \lambda^{(a)}(\beta)_{\beta_a+1}-\lambda^{(a)}(\beta)_{\beta_a + 2}, \qquad 1 \leq a \leq n$$
\end{defn}

\begin{lem}
    Given $\beta \in \Lat_\mu$, the prefix shapes $\lambda^{(a)}(\beta), 0 \leq a \leq n$ give a complete flag of partitions of $\mu$.
\end{lem}

\begin{proof}
For each $a$, the lattice condition gives
\[
    \lambda^{(a)}(\beta)_r
      =\#\{i\le a:\beta_i=r-1\}
      \ge
      \#\{i\le a:\beta_i=r\}
      =\lambda^{(a)}(\beta)_{r+1}
\]
for every $r\ge1$. Thus $\lambda^{(a)}(\beta)$ is a
partition after omitting trailing zeros. Moreover,
\[
    \sum_{r\ge1}\lambda^{(a)}(\beta)_r=a,
\]
since each of the first $a$ letters contributes to
exactly one part. Hence $\lambda^{(a)}(\beta)\vdash a$.

Passing from the prefix of length $a-1$ to the prefix
of length $a$ increases precisely the part indexed
by $r_a:=\beta_a+1$:
\[
    \lambda^{(a)}(\beta)_r
      =
      \begin{cases}
        \lambda^{(a-1)}(\beta)_r+1,
            &r=r_a,\\
        \lambda^{(a-1)}(\beta)_r,
            &r\ne r_a.
      \end{cases}
\]
Consequently,
$\lambda^{(a-1)}(\beta)\subsetneq\lambda^{(a)}(\beta)$,
and the difference consists of one box in row $r_a$.
Finally, $\lambda^{(0)}(\beta)=\varnothing$, and the
content condition gives $\lambda^{(n)}(\beta)=\mu$.
Therefore
\[
    \varnothing=\lambda^{(0)}(\beta)
      \subsetneq\lambda^{(1)}(\beta)
      \subsetneq\cdots
      \subsetneq\lambda^{(n)}(\beta)=\mu
\]
is a complete flag of partitions of $\mu$.
\end{proof}
As a quick corollary, the gaps $g_a(\beta)$ are therefore positive. Indeed, writing $r_a=\beta_a+1$,
the preceding proof gives
\[
\begin{aligned}
    g_a(\beta)
      &=\lambda^{(a)}(\beta)_{r_a}
          -\lambda^{(a)}(\beta)_{r_a+1}\\
      &=1+\lambda^{(a-1)}(\beta)_{r_a}
          -\lambda^{(a-1)}(\beta)_{r_a+1}
       \ge1,
\end{aligned}
\]
since $\lambda^{(a-1)}(\beta)$ is a partition.

\begin{example}\label{ex:prefix-flags-22}
Let $\mu=(2,2)$, so that
$\Lat_\mu=\{0011,0101\}$. The corresponding complete
flags and gaps are shown below.
\begin{center}
\begingroup
\ytableausetup{boxsize=0.9em,centertableaux}
\setlength{\tabcolsep}{10pt}
\renewcommand{\arraystretch}{1.5}
\begin{tabular}{@{}ccccc@{}}
\toprule
$a$
& $\lambda^{(a)}(0011)$
& $g_a(0011)$
& $\lambda^{(a)}(0101)$
& $g_a(0101)$\\
\midrule
0
& $\varnothing$
& ---
& $\varnothing$
& ---\\
\addlinespace[4pt]
1
& $\ydiagram{1}$
& 1
& $\ydiagram{1}$
& 1\\
\addlinespace[6pt]
2
& $\ydiagram{2}$
& 2
& $\ydiagram{1,1}$
& 1\\
\addlinespace[6pt]
3
& $\ydiagram{1,2}$
& 1
& $\ydiagram{1,2}$
& 1\\
\addlinespace[6pt]
4
& $\ydiagram{2,2}$
& 2
& $\ydiagram{2,2}$
& 2\\
\bottomrule
\end{tabular}
\endgroup
\end{center}

The two flags correspond to the standard Young tableaux
\[
    \begin{ytableau}
        3&4\\
        1&2
    \end{ytableau}
    \qquad\text{and}\qquad
    \begin{ytableau}
        2&4\\
        1&3
    \end{ytableau},
\]
respectively: label the box added at step $a$ by $a$.
\end{example}
\ytableausetup{boxsize=normal,nocentertableaux}

We now describe an insertion procedure, which we call \emph{flag insertion}, which takes a tuple $\zeta$ controlled by the gaps $g_a(\beta)$, and outputs a column-strict filling $\sigma \in \CST_\mu$. This algorithm is the inverse map $\psi_{\mathrm{inv}}^{-1}$ of
Bhattacharya--Ratheesh--Viswanath
\cite[Section~9.2, pp.~24--25]{bhattacharya2024whittaker}, specialized to standard content; we give a modified presentation adapted to our needs here.

\begin{defn}[Flag insertion]\label{def:flag-insertion}
Fix $\beta\in\Lat_\mu$ and a tuple
$\zeta=(\zeta_1,\ldots,\zeta_n)$ satisfying
\[
    0\le\zeta_a<g_a(\beta)
    \qquad(a\in[n]).
\]
Starting with the empty diagram of shape $\mu$, process
the rows from top to bottom and, within each row,
process its labels in decreasing order. Label $a$
is assigned to row $\beta_a+1$, where rows are numbered
from $1$ at the bottom.

A cell in this row is \emph{eligible} for $a$ if it is
empty and either has no cell directly above it or its
upper neighbor contains a label larger than $a$.
Number the eligible cells from right to left, starting
at $0$, and place $a$ in the cell numbered $\zeta_a$. See Example \ref{ex:flag-insertion-33} for a detailed computation.
\end{defn}

\Needspace{8\baselineskip}
A priori, it is not clear that all of the required choices exist; the next lemma proves this. Given $\sigma\in\CST_\mu$, we may reconstruct $\sigma$ using flag insertion; we call the resulting tuple the \emph{flag inversion code} of $\sigma$:
\[
    \zeta(\sigma)
       :=(\zeta_1(\sigma),\ldots,\zeta_n(\sigma)),
\]
where $\zeta_a(\sigma)$ is the number of eligible cells
strictly to the right of the cell containing $a$.
As with $\invvec(\sigma)$, the coordinates of
$\zeta(\sigma)$ are written in label order, not
in insertion order.

\begin{lem}\label{lem:flag-code}
Let $\beta\in\Lat_\mu$. At the step for label $a$,
flag insertion has exactly $g_a(\beta)$ eligible cells,
independent of the preceding choices. We therefore have a bijection
\begin{equation}\label{eq:flag-code-bijection}
    \prod_{a=1}^n\{0,1,\ldots,g_a(\beta)-1\}
       \longleftrightarrow\CST_\mu(\beta).
\end{equation}
Moreover, for every $\sigma\in\CST_\mu(\beta)$,
\begin{equation}\label{eq:flag-code-weight}
    |\zeta(\sigma)|
       =|\invvec(\sigma)|.
\end{equation}
\end{lem}

\begin{proof}
Suppose label $a$ belongs to row $r=\beta_a+1$.
All labels larger than $a$ in this row have already
been placed, leaving $\lambda^{(a)}(\beta)_r$ empty
cells. The row above has been completed, and its
$\lambda^{(a)}(\beta)_{r+1}$ entries at most $a$
forbid the cells immediately below them.

These forbidden cells are still empty: none could
have received an already placed label larger than
$a$ while preserving column strictness.
Every other empty cell is eligible, so we conclude the total number of eligible cells is
\[
    \lambda^{(a)}(\beta)_r
       -\lambda^{(a)}(\beta)_{r+1}
       =g_a(\beta),
\]
where there is no dependence on the preceding choices.
Since the gaps are positive, every allowed tuple produces a column-strict filling with row vector $\beta$. Reconstructing any such filling uniquely recovers its choices, proving the bijection.

For Equation~\eqref{eq:flag-code-weight}, recall that a \emph{reflected
inversion} \cite[Definition~8.1]{bhattacharya2024whittaker}
is a configuration
\[
    \begin{ytableau}
        \none & \none & c\\
        a & \none & b
    \end{ytableau}
    \qquad\text{with }b<a<c,
\]
where $a$ lies strictly to the left of $b$ in the same
row, and $c$ is directly above $b$. When the upper
cell is absent, set $c=+\infty$. Write $\operatorname{refinv}(\sigma)$
for the number of these configurations.

When $a$ is inserted, each skipped eligible cell to its
right eventually receives a smaller label $b$, since
labels in the row are processed decreasingly.
Eligibility says that its upper neighbor $c$ satisfies
$c>a$, or is absent. This cell therefore gives a reflected
inversion with left entry $a$.

Conversely, in any reflected inversion with left entry
$a$, the cell containing $b<a$ is still empty when $a$
is inserted. Since its upper neighbor satisfies $c>a$,
or is absent, it is eligible. Being to the right of $a$,
it is one of the skipped cells. Consequently,
\[
    \zeta_a(\sigma)
       =\#\{\text{reflected inversions with left entry }a\}.
\]
Summing over $a$ and applying
\cite[Proposition~8.1]{bhattacharya2024whittaker} gives
\[
    |\zeta(\sigma)|
       =\operatorname{refinv}(\sigma)
       =\inv(\sigma)
       =|\invvec(\sigma)|.\qedhere
\]
\end{proof}

\Needspace{10\baselineskip}
\begin{example}[Flag insertion for $(3,3)$]
\label{ex:flag-insertion-33}
Let $\mu=(3,3)$ and consider
\[
    \beta=(0,0,1,0,1,1),
    \qquad
    \zeta=(0,1,0,1,0,2).
\]
The corresponding complete flag is
\ytableausetup{boxsize=0.6em,centertableaux}
\[
    \varnothing
    \subseteq \ydiagram{1}
    \subseteq \ydiagram{2}
    \subseteq \ydiagram{1,2}
    \subseteq \ydiagram{1,3}
    \subseteq \ydiagram{2,3}
    \subseteq \ydiagram{3,3}.
\]
\ytableausetup{boxsize=normal,nocentertableaux}
and the gaps are
\[
    (g_1(\beta),\ldots,g_6(\beta))
       =(1,2,1,2,2,3).
\]
Notice $0 \leq \zeta_a < g_a(\beta)$ for all $a$, so this choice is valid. The bottom-row labels are $\{1,2,4\}$, and the top-row labels are $\{3,5,6\}$; flag insertion processes
them in the order
\[
    6,\ 5,\ 3,\ 4,\ 2,\ 1
\]
and produces the filling
\[
    \sigma=
    \begin{ytableau}
        6&3&5\\
        4&2&1\\
        \none[\scriptstyle 0]&
        \none[\scriptstyle 1]&
        \none[\scriptstyle 2]
    \end{ytableau},
\]
where the labels beneath the boxes are column coordinates. Notice $\sigma$ is indeed column-strict. The following table describes the insertion algorithm in detail:
\par
\begingroup
\small
\setlength{\tabcolsep}{6pt}
\renewcommand{\arraystretch}{1.3}
\setlength{\LTpre}{6pt}
\setlength{\LTpost}{6pt}
\setlength{\LTleft}{\fill}
\setlength{\LTright}{\fill}
\begin{longtable}{@{}cccccc@{}}
\toprule
Label $a$
& Row
& \shortstack{Eligible\\columns}
& \shortstack{Chosen\\column}
& $\zeta_a$
& Partial filling\\
\midrule
\endfirsthead
\multicolumn{6}{@{}l}{\itshape Flag insertion for $\mu=(3,3)$ (continued)}\\[3pt]
\toprule
Label $a$
& Row
& \shortstack{Eligible\\columns}
& \shortstack{Chosen\\column}
& $\zeta_a$
& Partial filling\\
\midrule
\endhead
\midrule
\multicolumn{6}{r@{}}{\footnotesize\itshape Continued on the next page}\\
\endfoot
\bottomrule
\endlastfoot
6 & 2 & $\{0,1,2\}$ & 0 & 2 &
$\begin{ytableau}
    *(pink)6
      & *(cyan)\phantom{1}
      & *(cyan)\phantom{1}\\
    \phantom{1} & \phantom{1} & \phantom{1}
\end{ytableau}$\\
\addlinespace[7pt]
5 & 2 & $\{1,2\}$ & 2 & 0 &
$\begin{ytableau}
    6 & *(pink)\phantom{1} & *(pink)5\\
    \phantom{1} & \phantom{1} & \phantom{1}
\end{ytableau}$\\
\addlinespace[7pt]
3 & 2 & $\{1\}$ & 1 & 0 &
$\begin{ytableau}
    6 & *(pink)3 & 5\\
    \phantom{1} & \phantom{1} & \phantom{1}
\end{ytableau}$\\
\addlinespace[4pt]
\midrule
4 & 1 & $\{0,2\}$ & 0 & 1 &
$\begin{ytableau}
    6 & 3 & 5\\
    *(pink)4 & \phantom{1} & *(cyan)\phantom{1}
\end{ytableau}$\\
\addlinespace[7pt]
2 & 1 & $\{1,2\}$ & 1 & 1 &
$\begin{ytableau}
    6 & 3 & 5\\
    4 & *(pink)2 & *(cyan)\phantom{1}
\end{ytableau}$\\
\addlinespace[7pt]
1 & 1 & $\{2\}$ & 2 & 0 &
$\begin{ytableau}
    6 & 3 & 5\\
    4 & 2 & *(pink)1
\end{ytableau}$\\
\addlinespace[4pt]
\end{longtable}
\endgroup

where the middle column is ineligible when inserting $4$ because $3$ occupies the box above it, and $3 < 4$.

For comparison, the original inversion triples of
$\sigma$ are
\[
    (-\infty,1,4),\qquad
    (-\infty,1,2),\qquad
    (-\infty,2,4),\qquad
    (4,5,6).
\]
Though they have the same total weight, in general, $\invvec(\sigma) \neq \zeta(\sigma)$. More precisely,
\[
    \invvec(\sigma)=(2,1,0,0,1,0),
    \qquad
    \zeta(\sigma)=(0,1,0,1,0,2),
\]
so $|\invvec(\sigma)|=|\zeta(\sigma)|=4$.
\end{example}

\begin{rem}
    The advantage of the flag insertion is that for a fixed row word $\beta$, the flag inversion codes are precisely the $\zeta$ that satisfy $0 \leq \zeta_a < g_a(\beta)$ for all $a$. This makes the subsequent enumeration of each of the filtration-by-$\yv$ components much easier. This property is not true for the Haglund basis $\Basis_\mu^{\mathrm{top}}$, as we shall see later.
\end{rem}

\subsection{Flag insertion Newton polynomials and the flag basis}

Reconstruct $\sigma\in\CST_\mu(\beta)$ by flag insertion.
At the step for $a$, list the eligible columns in decreasing order:
\[
    \widehat E_a^\sigma
       =\{\widehat e_{a,0}^\sigma>
           \widehat e_{a,1}^\sigma>\cdots>
           \widehat e_{a,g_a(\beta)-1}^\sigma\},
\]
and label the chosen column as
$c_\sigma(a)=\widehat e_{a,\zeta_a(\sigma)}^\sigma$.
Define
\begin{equation}\label{eq:companion-newton}
\begin{aligned}
    \widehat{\mathcal N}_\sigma(\xv)
      &:=\prod_{a=1}^n
           \prod_{0\le j<\zeta_a(\sigma)}
                (x_a-\widehat e_{a,j}^\sigma),\\
    \widehat{\mathcal G}_\sigma(\xv,\yv)
      &:=\widehat{\mathcal N}_\sigma(\xv)
             \yv^{\rowvec(\sigma)}.
\end{aligned}
\end{equation}
so that the roots in the $a$-factor of $\widehat{\mathcal N}_\sigma(\xv)$ are precisely the skipped eligible column numbers to the right
of the cell chosen for $a$. In particular,
\[
    \operatorname{in}_{\xv}\widehat{\mathcal G}_\sigma
       =\xv^{\zeta(\sigma)}\yv^{\rowvec(\sigma)}.
\]

\begin{thm}[Flag basis; Theorem~\ref{thmb:intro}]
\label{thm:gap-monomial-basis}
The monomials
\begin{equation}\label{eq:gap-monomial-basis}
    \widehat{\mathcal B}_\mu
       :=\left\{\xv^\zeta\yv^\beta:
          \beta\in\RowLat_\mu,\quad
          0\le\zeta_a<g_a(\beta)\text{ for all }a\right\}
\end{equation}
descend to a homogeneous basis of $V_\mu^{\tp}$.
\end{thm}

\begin{proof}
We apply the argument of Theorem~\ref{thm:ghbasis} analogously to
$\widehat{\mathcal G}_\sigma$.
Within a fixed row word, order fillings increasingly by
flag inversion code read in insertion order.
At the first differing step, the earlier filling chooses
one of the roots in the later Newton polynomial, so
$\widehat{\mathcal N}_\sigma(\mathbf p_\tau)=0$ when $\tau<\sigma$.
On the diagonal, evaluation gives
\[
    \widehat{\mathcal N}_\sigma(\mathbf p_\sigma)
       =\prod_a\prod_{j<\zeta_a(\sigma)}
          \bigl(c_\sigma(a)-\widehat e_{a,j}^\sigma\bigr)
       \ne0.
\]
By Lemmas~\ref{lem:column-alt-factorization}
and~\ref{lem:row-word-triangularity}, the pairing with the
$\mathcal A_\tau$ is block upper triangular in row-word order,
with diagonal blocks equal to $\Upsilon_\mu$ times these
invertible Newton matrices. The classes of the
$\widehat{\mathcal G}_\sigma$ are then independent in $\mathcal Q_\mu$.
By Lemma \ref{lem:flag-code}, they have the same
$\xv$-degree distribution as the Haglund interpolation basis.
The filtered-dimension argument of
Theorem~\ref{thm:ghbasis} therefore applies unchanged.
Passing to $\operatorname{gr}_{\xv}\mathcal Q_\mu\cong V_\mu^{\tp}$
gives a basis of initial monomials, and
\eqref{eq:flag-code-bijection} identifies it with
\eqref{eq:gap-monomial-basis}.
\end{proof}

\subsection{Filtration by \texorpdfstring{$\yv$}{y}-monomial and the fermionic formula}

We now give a graded $\C[\xv]$-submodule filtration of $V_\mu^{\mathrm{top}}$ by lexicographical order according to $\yv$-exponent. For $\beta\in\mathbb Z_{\ge0}^n$ with $|\beta|=N$, define
\begin{equation}\label{eq:row-word-filtration}
    \mathscr F_{\le\beta}V_\mu^{\tp}
       :=\sum_{\substack{\gamma\in\mathbb Z_{\ge0}^n,\ |\gamma|=N\\
                         \gamma\le_{\mathrm{lex}}\beta}}
           \C[\xv]\,[\yv^\gamma],
\end{equation}
where brackets denote classes modulo $J_\mu$.
Define $\mathscr F_{<\beta}$ similarly and set
\[
    \operatorname{gr}^{\mathrm{lex}}_\beta V_\mu^{\tp}
       :=\mathscr F_{\le\beta}V_\mu^{\tp}/
          \mathscr F_{<\beta}V_\mu^{\tp}.
\]

\begin{prop}\label{prop:row-word-filtration}
Both the Haglund basis of Theorem~\ref{thm:ghbasis} and the
flag basis of Theorem~\ref{thm:gap-monomial-basis} are compatible with~\eqref{eq:row-word-filtration} in the following sense: their elements with
row word at most $\beta$ form a basis of
$\mathscr F_{\le\beta}V_\mu^{\tp}$.
Consequently,
\begin{equation}\label{eq:row-subquotient-hilbert}
    \Hilb_q\!\left(\operatorname{gr}^{\mathrm{lex}}_\beta
                         V_\mu^{\tp}\right)
       =\begin{cases}
           \displaystyle\prod_{a=1}^n[g_a(\beta)]_q,
                &\beta\in\RowLat_\mu,\\[4pt]
           0,   &\beta\notin\RowLat_\mu,
         \end{cases}
\end{equation}
where $[m]_q:=1+q+\cdots+q^{m-1}$.
\end{prop}

\begin{proof}
Let $f=\xv^\alpha\yv^\gamma$ with $|\gamma|=N$, and expand
its class in $\mathcal Q_\mu$ in either basis $\{\mathcal{G}_\sigma\}, \{\widehat{\mathcal{G}}_\sigma\}$.
Lemma \ref{lem:row-word-triangularity} gives
\[
    \mathcal A_\tau(f)=0
       \qquad\text{if }\rowvec(\tau)>_{\mathrm{lex}}\gamma.
\]
Since the corresponding pairing matrix is block upper triangular
with invertible diagonal blocks, back substitution shows that
only basis elements with row words at most $\gamma$ occur.
Moreover, both bases $\{\mathcal{G}_\sigma\}, \{\widehat{\mathcal{G}}_\sigma\}$ are filtered, so only elements
of $\xv$-degree at most $|\alpha|$ occur. Passing to the
$\xv$-class in degree $|\alpha|$ expresses $[f]\in V_\mu^{\tp}$
using monomial basis elements with row words at most $\gamma$.
This proves spanning; independence follows from
the respective basis theorems.

We conclude the flag basis monomials with row word exactly $\beta$
give a basis of the subquotient. If $\beta$ is lattice,
their degree generating function is given by
\[
    \sum_{0\le\zeta_a<g_a(\beta)}q^{|\zeta|}
       =\prod_{a=1}^n[g_a(\beta)]_q.
\]
Otherwise there are no basis elements with that row word,
so the subquotient is zero.
\end{proof}

\begin{cor}
    We have
    \begin{equation}\label{eq:global-fermi}
    \Hilb_q(V_\mu^{\tp})
      =[m_{1^n}]W_\mu(X;q)
      =\sum_{\beta\in\RowLat_\mu}
          \prod_{a=1}^n[g_a(\beta)]_q,
\end{equation}
where $[m_{1^n}]$ extracts the coefficient in the monomial
symmetric-function basis. The flag basis $\widehat{\Basis}_\mu$ therefore witnesses the fermionic formula of the $m_{1^n}$ coefficient of $W_\mu(X;q)$.
\end{cor}

\begin{rem}
    In terms of the Gelfand--Tsetlin pattern terminology, consider a general interlacing chain $\varnothing=\lambda^{(0)}\subseteq\cdots\subseteq\lambda^{(m)}=\mu$, with $\lambda_r^{(a)}\ge\lambda_r^{(a-1)}\ge\lambda_{r+1}^{(a)}$, and set
    \[
    A_{a,r}:=\lambda_r^{(a)}-\lambda_r^{(a-1)},
    \qquad
    B_{a,r}:=\lambda_r^{(a-1)}-\lambda_{r+1}^{(a)}.
\]
The interlacing inequalities make $A_{a,r}$ and $B_{a,r}$ nonnegative. The fermionic weight of the chain is
\[
    \prod_{a=1}^{m}\prod_{r\ge1}
       \begin{bmatrix}A_{a,r}+B_{a,r}\\A_{a,r}\end{bmatrix}_q,
\]
where each $q$-binomial counts partitions fitting in an
$A_{a,r}\times B_{a,r}$ rectangle, weighted by their size;
see~\cite[Equations~(3.5) and~(4.1)]{bhattacharya2024whittaker}.
Since the flags associated with our $\beta$'s are saturated chains, $m=n$, and the only nonzero
$A_{a,r}$ at step $a$ is $A_{a,r_a}=1$, where $r_a=\beta_a+1$.
Also, $B_{a,r_a}=g_a(\beta)-1$, so the weight is precisely
$\prod_a[g_a(\beta)]_q$.
\end{rem}

\subsection{Enumeration of Haglund basis row-word fibers}

\begin{cor}[Corollary~\ref{corc:intro}]\label{cor:haglund-row-fiber}
For every $\beta\in\RowLat_\mu$, the classes of
\[
    \left\{\mathfrak m_\sigma:
               \sigma\in\CST_\mu(\beta)\right\}
\]
form a homogeneous basis of
$\operatorname{gr}^{\mathrm{lex}}_\beta V_\mu^{\tp}$.
In particular,
\begin{equation}\label{eq:row-fiber-fermi}
    \sum_{\sigma\in\CST_\mu(\beta)}
       q^{|\invvec(\sigma)|}
       =\prod_{a=1}^n[g_a(\beta)]_q.
\end{equation}
Equivalently, the fiber over each row-sorted standard tableau
has the same product enumerator.
\end{cor}

\begin{proof}
Proposition~\ref{prop:row-word-filtration} shows that the
Haglund basis is adapted to the same filtration as the
flag basis. Its elements with row word exactly $\beta$
therefore form a basis of the same subquotient, whose Hilbert
series is~\eqref{eq:row-subquotient-hilbert}.
\end{proof}

Equation~\eqref{eq:row-fiber-fermi} can be recognized as standard content specialization of \cite[Corollary~6.1(1), Equation~(6.1)]{bhattacharya2024whittaker}. Here it also describes the Hilbert series of the row-word filtration subquotients.

\begin{example}[Exponent bounds are not independent for the Haglund basis]
\label{ex:row-filtration-32}
Let $\mu=(3,2)$ and $\beta=(0,0,1,0,1)$.
The bottom-row labels are $\{1,2,4\}$ and the top-row
labels are $\{3,5\}$, so the common row-sorted tableau is
\[
    \begin{ytableau}
        3&5\\
        1&2&4
    \end{ytableau}.
\]
The gap sequence is
\[
    (g_1(\beta),\ldots,g_5(\beta))=(1,2,1,2,2).
\]
Consequently, the flag inversion codes are precisely
the independent choices
\[
    \zeta_1=\zeta_3=0,
    \qquad
    \zeta_2,\zeta_4,\zeta_5\in\{0,1\}.
\]

The two bases give the following representatives of
$\operatorname{gr}^{\mathrm{lex}}_\beta V_{(3,2)}^\mathrm{top}$.
The table lists only their $\xv$-factors before multiplying by $\yv^\beta=y_3y_5$:
\begin{center}
\begingroup
\setlength{\tabcolsep}{10pt}
\renewcommand{\arraystretch}{1.4}
\begin{tabular}{@{}cll@{}}
\toprule
$\xv$-degree
& Haglund basis
& Flag basis\\
\midrule
0 & $1$
  & $1$\\
1 & $x_1,\ x_2,\ x_3$
  & $x_2,\ x_4,\ x_5$\\
2 & $x_1^2,\ x_1x_2,\ x_1x_3$
  & $x_2x_4,\ x_2x_5,\ x_4x_5$\\
3 & $x_1^2x_2$
  & $x_2x_4x_5$\\
\bottomrule
\end{tabular}
\endgroup
\end{center}

The $\xv$-factors for the Haglund basis form the
disjoint union
\[
    \{x_1^a x_2^b:0\le a\le2,\ 0\le b\le1\}
    \;\sqcup\;
    \{x_1^a x_3:0\le a\le1\}.
\]
Their exponent bounds are therefore dependent:
when $x_3$ occurs, $x_2$ cannot occur and the exponent
of $x_1$ is at most $1$.
Therefore the exponents cannot be described by independent bounds.

Nevertheless, their degree generating function is
\[
    [3]_q[2]_q+q[2]_q
       =[2]_q^3
       =1+3q+3q^2+q^3.
\]
The flag basis gives the same product directly
from its three independent binary choices. We conclude
\[
    \Hilb_q\!\left(
        \operatorname{gr}^{\mathrm{lex}}_{00101}
            V_{(3,2)}^\mathrm{top}
    \right)
       =[2]_q^3.
\]
The product formula describes the graded enumeration
of both bases, but only the flag basis
realizes its factors as independent bounds on the
individual $\xv$-exponents.
\end{example}

\begin{rem}
The two fixed-row-word families are bases of the same
subquotient, but need not span the same subspace of
$V_\mu^{\tp}$ before quotienting by smaller row words.
\end{rem}

\section{Future directions}

Theorem \ref{thma:intro} was originally conjectured in the author's thesis in terms of a certain statistic called ``skip" on a certain set of generalized affine permutations that were in bijection with Haglund's fillings \cite[Conjecture~5.4.5, p.~83]{chou2023thesis}. (There the module is written as $V_{\lambda'}$, so its indexing partition satisfies $\lambda'=\mu$; its $\xv$- and $\yv$-exponent tuples are denoted $\operatorname{skipt}$ and $\operatorname{majt}$, respectively.) The construction uses the fibers of the ``shuffle map" defined by Carlsson and the author in \cite{carlsson2024descent}. The domain of this map, namely the disjoint union of its fibers over all values, has size $n!$ \cite[Section~3.1, p.~29]{chou2023thesis}, and suggests a path forward for answering Problem \ref{problem1}. However, the precise connection between orbit harmonics and the shuffle construction is unclear at the time of this writing, and remains speculative.

\begin{problem}
    Investigate the connection between orbit harmonics of $\mathcal{O}_\mu$ and the fibers of the shuffle map described in \cite{carlsson2024descent}. Use this connection to obtain a basis for the lower $t$-components.
\end{problem}

Another route is to consider using other formulas for $\widetilde{H}_\mu(X;q,t)$ as an oracle for discovering a combinatorial basis. Preliminary calculations in lower $t$-degree have been unwieldy and have not yet produced a compatible extension based on the Haglund--Haiman--Loehr formula for $\widetilde{H}_\mu(X;q,t)$. Unfortunately, other formulae for the modified Macdonald polynomials, such as those of Kaliszewski--Morse \cite{kaliszewski2019colorful} and Ayyer--Mandelshtam--Martin \cite{ayyer2023modified}, have yielded little progress as well; the filtrations of $\mathcal{R}_\mu$ tested so far have not produced a corresponding basis in lower $t$-degree. This suggests the following natural problem:

\begin{problem}
    Find a combinatorial formula for $\widetilde{H}_\mu(X;q,t)$ that is well-adapted to the geometry of $\mathcal{O}_\mu$.
\end{problem}

\Needspace{6\baselineskip}
Another possible way forward is to seek a simple extension of the methods of this paper:  

\begin{problem}
    Construct partial-column analogues of $\mathcal{A}_\tau$ on $\mathcal{O}_\mu$ and apply an analogous Newton-interpolation argument for the lower $t$-layers.
\end{problem}

Problem \ref{problem1} remains unsolved for arbitrary $\mu$, yet recent developments suggest exciting routes forward.

\bibliographystyle{amsalpha}
\bibliography{refs}

\end{document}